\documentclass
[
11pt
]
{amsart}

\usepackage{hyperref}
\usepackage{cite}
\usepackage{amssymb}
\usepackage{amsmath}
\usepackage{amsthm}
\usepackage{amsfonts}
\usepackage{graphicx}
\usepackage{bbm}
\usepackage{xcolor}
\usepackage[toc,page]{appendix}

\usepackage{mathrsfs}
 \usepackage{mathptmx}

\numberwithin{equation}{section}

\newtheorem{theorem}[equation]{Theorem}
\newtheorem{lemma}[equation]{Lemma}
\newtheorem{proposition}[equation]{Proposition}
\newtheorem{corollary}[equation]{Corollary}

\theoremstyle{definition}
\newtheorem{notation}[equation]{Notation}

\newtheorem{definition}[equation]{Definition}

\theoremstyle{remark}
\newtheorem{example}[equation]{Example}
\newtheorem{remark}[equation]{Remark}
\newtheorem{assump}[equation]{Assumption}

\DeclareSymbolFont{largesymbols}{OMX}{yhex}{m}{n}
\DeclareMathAccent{\widehat}{\mathord}{largesymbols}{"62}

\newcommand{\RR}
{\mathbf{R}}

\newcommand{\PP}
{\mathbf{P}}

\newcommand{\N}
{\mathbf{N}}

\newcommand{\OO}{\mathscr{O}}

\newcommand{\EE}{\mathbf{E}}

\newcommand{\f}{\bar{f}}

\begin{document}

\title[Gamma Calculus beyond Villani]{Gamma calculus beyond Villani and explicit convergence estimates for Langevin dynamics with singular potentials}

\author[Baudoin]{Fabrice Baudoin{$^{\star}$}}
\thanks{\footnotemark {$\star$} Research was supported in part by the Simons Foundation and NSF grant  DMS-1901315.}
\address{Department of Mathematics\\
University of Connecticut\\
Storrs, CT 06269, U.S.A.} \email{fabrice.baudoin@uconn.edu}

\author[Gordina]{Maria Gordina{$^{\dag }$}}
\thanks{\footnotemark {$\dag$} Research was supported in part by NSF grants DMS-1712427 and  DMS-1954264.}
\address{ Department of Mathematics\\
University of Connecticut\\
Storrs, CT 06269,  U.S.A.}
\email{maria.gordina@uconn.edu}

\author[Herzog]{David P. Herzog{$^{\ddag}$}}
\thanks{\footnotemark {$\ddag$} Research was supported in part by NSF grants DMS-1612898 and DMS-1855504}
\address{Department of Mathematics\\Iowa State University\\Ames, IA 50311, U.S.A.} \email{dherzog@iastate.edu}

\keywords{Langevin dynamics, hypocoercivity,  kinetic Fokker-Planck equation, Lyapunov function}

\subjclass{Primary60J60, 60H10, 35Q84; Secondary  35B40}


\begin{abstract}
We apply Gamma calculus to the hypoelliptic and non-symmetric setting of Langevin dynamics under general conditions on the potential.  This extension allows us to provide explicit estimates on the convergence rate (which is exponential) to equilibrium for the dynamics in a weighted $H^1(\mu)$ sense, $\mu$ denoting the unique invariant probability measure of the system.  The general result holds for singular potentials, such as the well-known Lennard-Jones interaction and confining well, and it is applied in such a case to  estimate the rate of convergence when the number of particles $N$ in the system is large.
\end{abstract}

\maketitle

\tableofcontents

\renewcommand{\contentsname}{Table of Contents}

\maketitle

\section{Introduction}

This paper studies convergence to equilibrium for second-order Langevin dynamics under general growth conditions on the potential. Although we are principally motivated by the case when the potential is singular, e.g. when the dynamics has repulsive forces and/or interactions, the results presented in this paper hold more generally.  In particular, our main result is that, given (very) basic structural and growth conditions on the potential, the dynamics relaxes to equilibrium exponentially fast in an explicitly measurable way.  The ``explicitness" of this result comes directly from the constants appearing in the growth conditions, which can all be readily estimated, and a local Poincar\'{e} constant for the invariant measure $\mu$.  This result is applied to the specific situation of a singular interaction and polynomial confining well to provide explicit estimates on the exponential convergence rate $e^{-\sigma}$ in terms of the number $N\geqslant 1$ of particles in the system.  We will see that $\sigma \geqslant c/(\rho \vee N^p)$, where $\rho>0$ is the local Poincar\'{e} constant for $\mu$ and $c>0, p\geqslant 1$ are constants that are independent of $N$.

Convergence to equilibrium for Langevin dynamics, sometimes called the kinetic Fokker-Planck equation, is a well-studied topic which has been investigated both from analytic and probabilistic perspectives.  The first known result in this circle of problems is due to Tropper in 1977~\cite{Tropper1977a} who proved mixing of the dynamics when the Hessian of $U$, denoted by $\nabla^2 U$, is bounded.  Tropper's result was subsequently improved in the papers of Talay~\cite{Talay2002} and Mattingly, Stuart and Higham~\cite{MattinglyStuartHigham2002}, both in 2002, where exponential convergence to equilibrium was obtained via the existence of a Lyapunov function of the form $H(x,v)+c x\cdot v$ provided the potential is ``polynomial-like".  Here
\begin{align}
H(x,v) = \frac{|v|^2}{2}+ U(x)
\end{align}
is the Hamiltonian of the system and $x$ and $v$ respectively denote the position and velocity vectors.  Additionally under this condition on the potential, in~\cite{Talay2002} Talay established exponential convergence to equilibrium in the topologies $H^k(\mu)$ for all $k\in \N$ and as well as exponential convergence in a weighted topology where the weight satisfies a Lyapunov-type condition.  These papers later inspired the work of Villani in~\cite{Villani2009}, which was applied to prove exponential convergence to equilibrium in $H^1(\mu)$ (and also $L^2(\mu)$) for $C^2$ potentials $U$ such that the measure $e^{-U} \, dx$ satisfies a Poincar\'{e} inequality as well as the growth condition $|\nabla^2 U| \leqslant C(1+ |\nabla U|)$.  This work was in some sense a combination of the ideas in~\cite{Talay2002, MattinglyStuartHigham2002} with the seminal works of H\'{e}rau and Nier~\cite{HerauNier2004} and Talay~\cite{Talay2002} where an appropriately chosen perturbation of the $H^1(\mu)$ norm is constructed in which the dynamics contracts.  The idea being proposed was that local smoothing of the dynamics in the sense of hypoellipticity~\cite{Hormander1967a} determined more global contractive properties like coercivity, hence the nomenclature ``hypocoercivity".  We refer the reader to \cite{DolbeaultMouhotSchmeiser2015, IacobucciOllaStoltz2019} for either different methods for proving, or other applications of hypocoercivity.

Even though Villani's result allows for a general class of potentials, which in particular subsumes the class of potentials treated in~\cite{MattinglyStuartHigham2002, Talay2002}, the growth condition $|\nabla^2 U| \leqslant C(1+ |\nabla U|)$ is not satisfied by potentials with singularities.  This is because the singularity becomes ``stronger" with each additional derivative.  Subsequently, the work of Conrad and Grothaus~\cite{ConradGrothaus2010} and Grothaus and Stilgenbauer~\cite{GrothausStilgenbauer2015} extended Villani's result using the hypocoercive approach to prove ergodicity and establish a rate of convergence (at least as fast as polynomial) for singular potentials satisfying the same Poincar\'{e} condition but the weaker growth condition:
\begin{align*}
\text{For all }\,c>0\,\text{ there exists }\,D(c)>0\,\text{ such that  }|\nabla^2 U| \leqslant c|\nabla U|^2 +D(c).
\end{align*}
Under this same growth condition, convergence to equilibrium at an exponential rate was established in the papers~\cite{CookeHerzogMattinglyMcKinleySchmidler2017, HerzogMattingly2017} by the construction of an explicit Lyapunov function of the form $\exp(\delta H+ \psi)$ where $\delta >0$ is a constant, $H$ is the Hamiltonian and $\psi$ an appropriately constructed perturbation.  The ideas used in these papers utilized large state space asymptotics of the dynamics previously developed in~\cite{AthreyaKolbaMattingly2012, HerzogMattingly2015a, HairerMattingly2009} to build the Lyapunov function.  During a similar time frame, the work of Cattiaux, Guillin, Monmarch\'{e} and Zhang~\cite{CGMZ_2017} proved entropic convergence at a geometric rate employing a weighted log-Sobolev inequality satisfied by potentials under the growth condition
\begin{align}
\label{eqn:CGMZ}
|\nabla^2 U| \leqslant C U^{2\eta} \leqslant C' U^{2\eta +1} \leqslant |\nabla U|^2
\end{align}
outside of a compact domain in space for some constants $C, C', \eta \geqslant 0$.  While this condition does allow for a singular confining potential, as opposed to the growth conditions used in~\cite{ConradGrothaus2010, GrothausStilgenbauer2015, HerzogMattingly2017} it does not afford the flexibility of a difference in asymptotic behavior between the confining well and the interaction potential.  A simple example for which this condition is not satisfied is provided in Section~\ref{sec:weight}.

In terms of explicitly quantifying convergence rates to equilibrium for Langevin dynamics, results for a broader class of potentials than by Villani in~\cite{Villani2009} are much rarer, but significant recent progress has been made by
Eberle, Guillin and Zimmer~\cite{EGZ_17, EberleGuillinZimmer2019} using a direct coupling approach and assuming that the derivatives of $U$ are bounded.  Although seemingly limited by this boundedness assumption, a path is provided by Zimmer~\cite{ZimmerPhDThesis} to extend the results to more general potentials using a Lyapunov condition.  However, the actual dependence of the constants in the Lyapunov condition on the growth conditions satisfied by the potentials remains largely unknown except in the cases where either $U$ satisfies~\eqref{eqn:CGMZ} as in \cite{CGMZ_2017} or behaves  like a quadratic potential as in \cite{EGZ_17}.  Moreover, the relationship between these two commonly employed approaches to tackling the convergence question, analytic versus probabilistic, is not clear in this setting where the generator $L$ is hypoelliptic and non-symmetric with respect to the $L^2(\mu)$ inner product.

In this paper, we provide a different approach to estimating explicit convergence rates within a wide class of potentials for Langevin dynamics which is more in line with the approach of Villani in \cite{Villani2009} and Baudoin in \cite{Baudoin2017a}.  In particular, our main general result directly extends the work of Villani using Gamma calculus.  We then use the results of this approach to estimate how the rate of convergence depends on the growth conditions satisfied by the potential and, in the specific case of a singular interaction and polynomial confining well, we estimate the dependence of the convergence rate on the number $N$ of particles in the system.  An interesting consequence of the framework provided is that the construction of the norm in which the dynamics contracts is very similar in spirit to the typical Harris construction in \cite{HairerMattingly2011b, MeynTweedie1993a}.  That is, to couple two different initial conditions, one must wait until both processes return to the ``center" of space.  Then, once both processes enter this domain, the noise provides the mechanism for mixing.  Here we will build the norm by making use of Lyapunov structure of a slightly different variety outside of the center and a local Poincar\'{e} inequality satisfied by the invariant measure $\mu$ in the center.  It should be noted that this approach is different than the one outlined by Monmarch\'{e} in \cite{Monmarche2015, Monmarche2019} in that we do not modify the natural carr\'{e} du champ operator associated to the generator $L$ of the diffusion.

The organization of this paper is as follows.  In Section~\ref{sec:statmainres}, we set up notation, introduce terminology and state the main general result.  There we also present concrete examples of potentials and estimates on the convergence rates. In Section~\ref{sec:weightone} and Section~\ref{sec:weight}, we gradually construct the (weighted) $H^1$ metric in which the dynamics contracts, explaining why the presence of the weight is natural and also explaining how the Villani condition on $U$ arises in the process.  The subsequent section, Section~\ref{sec:LFs}, shows that the Lyapunov structure on which the results of Section~\ref{sec:weightone} relies is not hard to verify and estimate quantitatively.

\section{Mathematical setting, main results and examples}
\label{sec:statmainres}

In this section, we fix notation, terminology and our assumptions.  Then we state the main results to be proved in this paper.  Following the statements of these results, we end the section with concrete examples of potentials and some explicit convergence rates.

\subsection{Setting and basic assumptions}  Throughout the paper we study the following Langevin stochastic differential equation
\begin{align}
\label{eqn:main}
dx(t)&= v(t) \, dt \\
dv(t)&= -\gamma v(t) \, dt - \nabla U(x(t)) \, dt + \sqrt{2\gamma T} \, dB(t), \notag
\end{align}
where
\begin{align*}
& x(t)=(x_1(t), x_2(t), \ldots, x_N(t))\in (\RR^k)^N,
\\
& v(t)=(v_1(t), v_2(t), \ldots, v_N(t)) \in (\RR^k)^N
\end{align*}
denote the position and velocity vectors, respectively, of $N$-particles evolving on some subset of $\RR^k$.  The parameters $\gamma, T>0$ are the friction and temperature constants, respectively, while $B(t)$ is a standard $(Nk)$-dimensional Brownian motion defined on a probability space $(\Omega, \mathcal{F}, \PP)$.  The function $U:(\RR^k)^N\rightarrow [0,  \infty]$ is the potential function. It encodes both environmental forces acting on all particles as well as particle interactions.  We will allow $U$ to take the value $\infty$ if the potential function has singularities, as in the case when $U$ has a Lennard-Jones interaction.  More specifically, $U$ takes the value $\infty$ at the point(s) in $(\RR^k)^N$ of singularity.  Thus if $U$ is nonsingular, as in the case of a standard polynomial interaction and confining well, then $U$ never takes this value.

\begin{notation} We denote by $\OO$ the  subset of $(\RR^k)^N$ on which the position vector $x(t)$ lives, namely,
\begin{align}
\OO := \{ x\in (\RR^k)^N \, : \, U(x) < \infty\}.
\end{align}
By $\mathcal{X}$ we denote the state space
\begin{align}
\mathcal{X}:= \OO \times (\RR^k)^N
\end{align}
and let
\begin{align}
d:=Nk
\end{align}
denote the spatial dimension parameter.

\end{notation}

\begin{assump}[Basic Structure]
\label{assump:1}
~
\begin{itemize}
\item $U\in C^\infty(\OO; [0, \infty))$.
\item The set $\OO$ is connected.  Moreover, for every $n\in \N$ the open set
\begin{align*}
\OO_n = \{ x\in (\RR^k)^N \, : \, U(x) < n \}
\end{align*}
has compact closure.
\item The integral $\int_\OO e^{-\frac{1}{T} U} \, dx $ is finite.
 \end{itemize}
\end{assump}
\begin{assump}[Growth Condition]\label{assump:2}
There exists a constant $\kappa^{\prime \prime} >0$ such that
\begin{align}
\label{eqn:gbound1}
|\nabla^2 U(x) y| \leqslant \frac{1}{16Td} |\nabla U(x) |^2|y| + \kappa^{\prime \prime}|y|
\end{align}
for all $x\in \mathcal{O}, y \in (\RR^k)^N$.  Here $\nabla^2 U$ denotes the Hessian matrix of $U$. Furthermore, there exist constants $c_{0}, c_\infty, d_0, d_\infty >0$ and $\eta_0 \in( -\infty, -1) \cup (0, \infty)$, $\eta_\infty>1$ such that
\begin{align}
\label{eqn:gbound2}
c_\infty U^{2-\frac{2}{\eta_\infty}} - d_\infty \leqslant |\nabla U|^2 \leqslant c_0 U^{2+\frac{2}{\eta_0}} + d_0.
\end{align}
\end{assump}

\subsection{Explanation of the assumptions}  We start by observing that the first assumption, Assumption~\ref{assump:1}, is a nominal requirement to ensure that Equation \eqref{eqn:main} has unique pathwise solutions.  Indeed, using the standard fixed-point argument, it is not hard to show that Assumption~\ref{assump:1} implies that Equation~\eqref{eqn:main} has unique local pathwise solutions $(x(t), v(t))$ evolving on the state space $\mathcal{X}$.

More precisely, denote by $B_n$, $n\in \N$,   the open ball of radius $n$ in $(\RR^k)^N$ centered at the origin and define stopping times $\tau_n$ and $\tau$ by
\begin{align*}
& \tau_n = \inf \{ t\geqslant 0: (x(t), v(t)) \notin \OO_n \times B_n\},
\\
& \tau= \lim_{n\rightarrow \infty} \tau_n.
\end{align*}
Then pathwise solutions exist and are unique for all times $t< \tau$~ (e.g. \cite{KhasminskiiBook2012, Rey-Bellet2006}).  However, by using the Hamiltonian
\begin{align}
H(x,v) := \frac{|v|^2}{2}+ U(x)
\end{align}
as a basic type of Lyapunov function,  it can be shown that
\begin{align}
\PP_{(x,v)}\{ \tau < \infty \} =0 \,\,\text{ for all } \,\, (x,v) \in \mathcal{X}.
\end{align}
Hence, Equation~\eqref{eqn:main} has unique pathwise solutions for all finite times $t\geqslant 0$ almost surely evolving on the state space $\mathcal{X}$.  Moreover, the pathwise solutions given by Assumption~\ref{assump:1} are Markovian, and we denote by $\left\{ (P_t) \right\}_{t\geqslant 0}$ the associated Markov semigroup acting on the space $C_b^2(\mathcal{X}; \RR)$ of bounded $C^2$ real-valued functions on $\mathcal{X}$.  Note that the semigroup $\left\{ P_t \right\}_{t\geqslant 0}$ has its infinitesimal generator $L$ given by
\begin{align}
L = v \cdot \nabla_x - \gamma v \cdot \nabla_v - \nabla U \cdot \nabla_v + \gamma T \Delta_v.
\end{align}

Another important consequence of Assumption~\ref{assump:1} is that it implies that the canonical Boltzmann-Gibbs measure $\mu$ given by
\begin{align}
\label{eqn:Gibbs}
& \mu(dx \,dv) = \frac{1}{\mathcal{N}}  e^{-\frac{1}{T} H(x,v)} \, dx \, dv,
\\
& \mathcal{N}= \int_\mathcal{X} e^{-\frac{1}{T}H(x,v)} \, dx \, dv, \notag
\end{align}
is an invariant probability measure under the dual action of the Markov semigroup $P_{t}$; that is, $\mu(\mathcal{X})=1$ and for all Borel sets $A$ in $\mathcal{X}$  and all $t>0$ we have
\begin{align*}
\mu P_t(A):= \int_\mathcal{X} \mu(dy) P_t(y, A) = \mu(A),
\end{align*}
where $P_t(y, \, \cdot \,)$ denotes the Markov transition kernel, whose existence is also guaranteed by Assumption~\ref{assump:1}.
\begin{remark}
It is possible to replace Assumption~\ref{assump:1} by a weaker regularity hypothesis on the potential $U$.  With only minor adjustments, all of the general results stated here hold provided $U\in C^2(\mathcal{O})$, but we maintain the assumption that $U\in C^\infty(\mathcal{O})$ to keep the presentation as simple as possible and because all of the potentials we are interested in satisfy this assumption.  In certain scenarios, one can treat potentials less regular than $C^2(\mathcal{O})$.  See, for example, the papers~\cite{ConradGrothaus2010, GrothausStilgenbauer2015}. However, in such cases, greater care must be taken to make sense of solutions of~\eqref{eqn:main}.
\end{remark}

\begin{remark}
There are natural cases where the set $\mathcal{O}$ is not connected.  For example, if $N=k=1$ and $U: \RR\rightarrow (0, \infty]$ is defined by
\begin{align*}
U(x) = \begin{cases}
|x|^a + \frac{1}{|x|^b}& \text{ if } x\neq 0
\\
\infty & \text{ if } x=0
\end{cases}
\end{align*}
for some constants $a>1$ and $b>0$, then the set $\mathcal{O}$ has two disjoint connected components, $\RR_{>0}$ and $\RR_{<0}$.  In this case, the potential $U$ models a ``hard wall" at $0$ that the single particle cannot penetrate.  As a result, solutions started on either side of $0$ generate mutually exclusive invariant measures.  This is not a weakness in Assumption~\ref{assump:1}, for it allows one to work within each connected component of $\mathcal{O}$ by setting either $U(x)= \infty$ when $x\leqslant 0$ or $U(x)=\infty$ when $x\geqslant 0$.
\end{remark}

The second assumption, Assumption~\ref{assump:2}, is a basic growth condition on the potential.  It is satisfied by a wide class of potentials, including those which contain Lennard-Jones interactions, for general temperatures $T>0$ and general dimensions $d=Nk$, and logarithmic singular interactions, provided the temperature $T>0$ is sufficiently small depending on the dimension.
\begin{remark}\label{r.2.3}
One should think of threshold of $1/(16Td)$ in Assumption~\ref{assump:2} as a quantitative version of the hypothesis made in \cite{HerzogMattingly2017}, where it is assumed that for any sequence $x_n \in \mathcal{O}$ with $U(x_n)\xrightarrow[n \to \infty]{} \infty$ we have
\begin{align*}
\frac{|\nabla^2 U(x_n)|}{|\nabla U(x_n)|^2} \xrightarrow[n \to \infty]{}  0.
\end{align*}
Thus under such an assumption, the bound~\eqref{eqn:gbound1} is always satisfied for some $\kappa''>0$.
\end{remark}

\begin{remark}
With the appropriate Lyapunov function, it is possible to improve the threshold $1/(16Td)$ in some cases.  See, for example, the recent work~\cite{lu_19} where Coloumb interaction potentials are treated in detail.  However, when $|\nabla^2 U|$ is of the same order as $|\nabla U|^2$ asymptotically when $U\rightarrow \infty$, then $U$ behaves like a logarithmic function.  Therefore in such cases, our assumptions are near the boundary where integrability of the Gibbs density fails and thus major structural changes are taking place in the measure.
\end{remark}
%
%
%

%

\subsection{The main results}
Before stating the main results, we need the following basic notation.
\begin{notation}
For a parameter $\zeta >0$, we define the modified gradient operator $\nabla_\zeta$ by
\begin{align}
\label{eqn:modgrad}
\nabla_\zeta :=\zeta^{-1} ( \nabla_v,  \nabla_x-c(\gamma) \nabla_v),
\end{align}
where
\begin{align}
\label{def:cgamma}
c(\gamma):= \frac{\gamma}{2}+ \sqrt{\frac{\gamma^2}{2}+1}.
\end{align}
Also for any $W\in L^1(\mu)$, we let $\mu_W$ denote the following (finite) weighted measure on Borel subsets of $\mathcal{X}$
\begin{align}
\label{eqn:wborel}
\mu_W(d x \, dv) := W(x,v) e^{-\frac{1}{T}H(x,v)} \, dx \, dv.
\end{align}
We denote  by $H^1_{\zeta,W}$ the space of weakly differentiable functions $f: \mathcal{X}\rightarrow \RR$ with
\begin{align}
\|f\|^2_{ \zeta, W}:=  \int_\mathcal{X}  f^2 \, d\mu_W   + | \nabla_\zeta f|^2 \, d\mu < \infty.
\end{align}
\end{notation}
As we will now see, for the appropriate choice of $W\in L^1(\mu)$ and $\zeta >0$, the distance $\| \cdot \|_{\zeta, W}$ is the type of norm in which the semigroup $P_t$ is contractive for all $t>0$.

\begin{theorem}
\label{thm:mainconv}
Suppose that the potential $U$ satisfies Assumption~\ref{assump:1} and Assumption~\ref{assump:2}.  Then there exists an explicit function $W\in C^\infty(\mathcal{X}; [1, \infty))\cap L^1(\mu)$ and explicit constants $\sigma, \zeta >0$ such that for all $f \in H^1_{\zeta, W}$ satisfying $\int_\mathcal{X} f \, d\mu =0$ we have
\begin{align}
\label{eqn:mainconvest}
\|P_t f\|_{\zeta, W}^2 \leqslant e^{-\sigma t} \| f\|_{\zeta, W}^2
\end{align}
 for all $t\geqslant 0$.
 \end{theorem}

\begin{remark}
Below we give an expression for the function $W$ above as well as estimates on the parameters $\sigma, \zeta>0$, thus making the bound~\eqref{eqn:mainconvest} explicit.  It should be noted that we did not assume that $\mu$ satisfies the Poincar\'{e} inequality on $\mathcal{X}$, although it is implicitly implied by the first condition in Assumption~\ref{assump:2}.  See, for example, the work of Villani~\cite{Villani2009}.  Because this bound is ``near" the boundary where the Poincar\'{e} inequality is satisfied, this result is ``closer" to establishing the equivalence for Langevin dynamics between exponential convergence to equilibrium and $\mu$ satisfying the Poincar\'{e} inequality on $\mathcal{X}$.
\end{remark}

To introduce a valid choice of $W$ and estimate the constants $\sigma, \zeta>0$ above, we need another definition.
\begin{definition}
\label{def:LPI}
Let $A\subseteq \mathcal{X}$ be a Borel set.  If $\nu$ is a positive Borel measure on $\mathcal{X}$ and $\nu(A)\in (0, \infty)$, we say that $\nu$ satisfies the \emph{local Poincar\'{e} inequality on} $A$ if there exists a constant $\rho=\rho(\nu, A)>0$ such that
\begin{align*}
\int_A f^2 \, d\nu \leqslant \rho \int_A |\nabla f|^2 \, d\nu  + \frac{1}{\nu(A)}\Big( \int_A  f \, d\nu \Big)^2
\end{align*}
for all $f\in H^1(\nu)$.
\end{definition}

\begin{remark}
\label{rem:connected}
Since the Boltzmann-Gibbs density $\mathcal{N}^{-1}e^{-\frac{1}{T} H}$ is strictly positive on $\mathcal{X}$, the measure $\mu$ defined by \eqref{eqn:Gibbs} satisfies the local Poincar\'{e} inequality on any compact and connected set $J\subseteq \mathcal{O}$.  Employing Assumption~\ref{assump:1} and Assumption~\ref{assump:2}, it follows that any set of the form
\begin{align*}
\{ (x,v) \in \mathcal{X} \, : \, |v|^2 \leq A, \, U(x) \leq B \}, \qquad A \geq 0, \,\,\, B >\big(\tfrac{d_\infty}{c_\infty}\big)^{\frac{1}{2-2/\eta_\infty}}\end{align*}
is both compact and connected.  Indeed, compactness is immediate.  Moreover, connectedness follows by the Mountain Pass Theorem~\cite{Kat_94}, a generalization of Rolle's theorem, since $\mathcal{O}$ is connected by Assumption~\ref{assump:1}, and since 
\begin{align*}
|\nabla U(x)| >0 \,\,\,\text{ whenever } \,\, \,U(x) >\big(\tfrac{d_\infty}{c_\infty}\big)^{\frac{1}{2-2/\eta_\infty}}\end{align*}
by relation~\eqref{eqn:gbound2}.  
\end{remark}

Let $U$ satisfy Assumption~\ref{assump:1} and Assumption~\ref{assump:2}.  We introduce constants
\begin{align}
& R_1 :=  \left\{ \frac{d_\infty}{c_\infty} +
\frac{(40e^4+4) Td (\kappa''+1) \vee 92 \gamma^2 Td) }{c_\infty} \right\}^{\frac{1}{2-\frac{2}{\eta_\infty}}}, \label{eqn:constr1r2}
\\
& R_2 :=R_1+ 32Td, \notag
\end{align}
where we recall that $c_0, d_0, c_\infty, d_\infty, \kappa''>0$ and $\eta_0 , \eta_\infty$ are the constants in Assumption~\ref{assump:2}, and $T>0, d=Nk$ are the temperature and dimensionality constants, respectively.
Define $\kappa'=1/(16Td)$ and for $r\geqslant 0$ set
\begin{align}
D(r)= \frac{2(c_0 \kappa')^2 r^{4+ \frac{4}{\eta_0}} + 2(d_0 \kappa')^2 + (\kappa'')^2}{\gamma^2 T}+ \frac{1}{2T}.
\end{align}
Introduce the compact and connected subset $K\subseteq\mathcal{X}$ (cf. Remark~\ref{rem:connected} and relation \eqref{eqn:constr1r2}) given by
\begin{align}
\label{eqn:Kcomp}
K = \{(x,v) \in \mathcal{X}\,: \,  |v|^2 \leqslant ( 20 e^4+2)Td \} \cap \{ (x,v) \in \mathcal{X} \, : \, U \leqslant R_2\} .
\end{align}
Since $\mu$ satisfies the local Poincar\'{e} inequality on $K$, throughout we let $\rho_K>0$ denote the associated Poincar\'{e} constant as in Definition~\ref{def:LPI}.  Set
\begin{align*}
\rho'_K = \frac{(4 c^2(\gamma) +4) \rho_K}{\gamma}.
\end{align*}
Next, let $\alpha =\frac{\gamma  Td }{4R_2}$, $\beta = \frac{5\gamma Td}{4 R_2} e^{4}$ and $\lambda_0, \lambda >0$ be such that
\begin{align*}
\lambda_0 \geqslant R_2 \log (D(\lambda_0)) + R_2 \log( \beta \rho_K' +1)
\end{align*}
and
\begin{align}
\label{eqn:lambda}
\lambda\geqslant  (\beta \rho'_K +1) D(\lambda_0) . \end{align}
Define
\begin{align}
\label{eqn:zetasigma}
\zeta^2 = \frac{2}{1+ \beta \rho_K'} \qquad \text{ and } \qquad \sigma = \frac{\alpha}{2(1+\lambda)} \wedge \frac{\gamma}{1+ \beta \rho_K' }.
\end{align}
Next, let $h\in C^\infty([0,\infty); [0,1])$ be any function satisfying
\begin{align*}
h(q)= \begin{cases}
1 & \text{ if } q\geqslant R_2 \\
0 & \text{ if } q\leqslant R_1
\end{cases}
\,\,\, \text{ and } \,\,\,
|h'| \leqslant \frac{2}{R_2 - R_1}= \frac{1}{16 T d}.
\end{align*}
Define
\begin{align*}
\psi(x,v) =\begin{cases}
 \displaystyle{-\frac{3\gamma Td }{2R_2}    \, h(U(x))    \frac{v\cdot \nabla U(x)}{|\nabla U(x)|^2} }& \text{ if } U(x) \geqslant R_1 \\
 0 & \text{ otherwise}
 \end{cases}
\end{align*}
and
\begin{align}
\label{eqn:Vlyap}
V(x,v) = \exp\bigg(\frac{H(x,v)}{R_2} + \psi(x,v)\bigg).
\end{align}

\begin{corollary}
\label{cor:mainexpl}
Suppose that the potential $U$ satisfies Assumption~\ref{assump:1} and Assumption~\ref{assump:2}. Then in the statement of Theorem~\ref{thm:mainconv} we may choose $\sigma, \zeta >0$ as in~\eqref{eqn:zetasigma} and the weight $W$ to be $W= e^{1} V + \lambda$ where $V$, $\lambda\geqslant 1$, and $K$ are as in~\eqref{eqn:Vlyap}, \eqref{eqn:lambda} and~\eqref{eqn:Kcomp}, respectively.
\end{corollary}

To see the utility of this result and the arguments which establish it, later we will estimate these constants in the following three concrete examples.
\begin{example}
\label{ex:single}
As a first example, consider the simple ``single-well" quadratic potential $U:\RR^d\rightarrow[0, \infty)$ given by $U(x)=|x|^2/2$.
Then $U$ clearly satisfies Assumption~\ref{assump:1}.  One can also easily show that $U$ satisfies Assumption~\ref{assump:2}.  In Section~\ref{sec:weightone}, we will see that in the bound~\eqref{eqn:mainconvest} we can choose $W= 1$,
\begin{align*}
\zeta^2 = \frac{3}{\gamma^2 T} + \frac{\big(\frac{\gamma}{2}+ \sqrt{\frac{\gamma^2}{4}+1} \big)^2}{2T} + \frac{1}{4T}\qquad \text{ and } \qquad \sigma = \frac{\gamma}{4}\min\{ 1, \frac{1}{T\zeta^2}\}.
\end{align*}
In this simple case as well as in the example that follows, we do not need to directly apply Theorem~\ref{thm:mainconv}.  Rather, we will be able to apply a corollary of its proof which allows one to get around unnecessary estimates needed in the case of a singular potential.
\end{example}

\begin{example}
\label{ex:double}
Next consider the double-well potential $U:(\RR^k)^N\rightarrow [0, \infty)$ given by $U(x)= (|x|^2-1)^2/4.$  Define constants $\kappa_0>0$ and $\kappa_0'>0$ by
\begin{align*}
\kappa_0 = \frac{\gamma}{2\sqrt{T+Tc^2(\gamma)}} \,\,\text{ and }\,\, \kappa_0' = \frac{27}{\kappa_0^2} +2
\end{align*}
where $c(\gamma)= \frac{\gamma}{2}+ \sqrt{\frac{\gamma^2}{4}+1}$ is as in~\eqref{def:cgamma}.  One can check that $U$ satisfies a Poincar\'{e} inequality on $(\RR^k)^N$  with some constant $\rho >0$.  Define
\begin{align*}
M^2=\frac{2\gamma^2 T \kappa_0 d}{4 c^2(\gamma)} + \frac{\sqrt{2d} \kappa_0' \gamma^2}{4 c^2(\gamma)} + (\kappa_0')^2.
\end{align*}
In Section~\ref{sec:weightone}, we will see that in this case the bound~\eqref{eqn:mainconvest} is satisfied for $W=1$,
\begin{align*}
& \zeta^2 = \frac{2+M^2}{\gamma^2 T} + \frac{\big(\frac{\gamma}{2}+ \sqrt{\frac{\gamma^2}{4}+1} \big)^2}{2T} + \frac{1}{4T},
\\
& \sigma = \frac{\gamma}{4}\min\{ 1, 1/\rho\zeta^2\}.
\end{align*}

\end{example}

\begin{example}
\label{ex:LJ}
Using the notation $x=(x_1, x_2, \ldots, x_N) \in (\RR^k)^N$, where each $x_i =( x_i^1, \ldots , x_i^k )$ belongs to $\RR^k$, we next consider a singular potential $U: (\RR^k)^N\rightarrow [0, \infty]$ of the form
\begin{align}
U(x)= \sum_{i=1}^N U_E(x_i)+ \sum_{i<j} U_I(x_i-x_j)
\end{align}
with $U_E\in C^\infty(\RR^k; [0, \infty))$ and $U_I:\RR^k\rightarrow [0, \infty]$ satisfying
\begin{align}
U_E(q)= A |q|^{a} \qquad \text{ and } \qquad U_I(q)= \begin{cases}
B |q|^{-b} &\text{ if } q\neq 0\\
\infty & \text{otherwise}
\end{cases}
\end{align}
where $A, B, b >0$ are constants and $a\geqslant 2$ is an even integer.  In the analysis that follows, we could certainly incorporate lower-order terms in both $U_E$ and $U_I$ above, but the estimates become unnecessarily complicated.  Thus we stick to the form above under these assumptions on $U_E$ and $U_I$.  Note that $U$ clearly satisfies Assumption~\ref{assump:1} for $k\geqslant 2$.  If, however, $k=1$ then we need to slightly modify the definition of the potential $U$ so that the domain of the particles $\mathcal{O}$ is connected.  That is, in this case we also set $U=\infty$ if the relation
\begin{align*}
x_1 < x_2 < \cdots < x_N
\end{align*}
is NOT satisfied.  This means that the particles must remain in the same ordering because they cannot pass one another when $k=1$.

In the Appendix, we will prove the following result showing that $U$ satisfies Assumption~\ref{assump:2} with explicit estimates on the constants in the assumption.
\begin{proposition}
\label{prop:singularconst}
$U$ as above satisfies Assumption~\ref{assump:2} with $\eta_0 = b$, $\eta_\infty=a$,
\begin{align*}
\kappa'' & =N^{5-\frac{8}{a}}Aa(a-1)k\left(\frac{128 (a-1) k^2 T}{Aa} \right)^{\frac{a-2}{a}} \\&+ N^{10+\frac{16}{b}}4 Bb(b+3)k \left( \frac{512 (b+3)k^2 T}{Bb}\right)^{\frac{b+2}{b}} +\frac{A^2a^2}{8 N^2 k T}  + \frac{B^2b^2 N^{2b+4}}{8kT},
\end{align*}
and choice of constants $c_0, d_0, c_\infty, d_\infty$ given by
 \begin{align*}
& c_0 = N^3\frac{4b^2}{B^{\frac{2}{b}}},
\\
& d_0 =N^{1- \frac{2(a-1)b}{a+b}}2A^2 a^2 \left( \frac{ A^{\frac{2}{b}}b^2}{B^{\frac{2}{b}} a^2}\right)^{\frac{(a-1)b}{a+b}},
\\
& c_\infty = \frac{A^{\frac{2}{a}} a^2}{ 2^{5-\frac{2}{a}} N^{5-\frac{2}{a}}},
\\
& d_\infty = N^{ \frac{b(6a-2)(a-1)}{a(a+b)}+ \frac{2}{a}-1}  \frac{A^{\frac{2}{a}} a^2 B^2}{ 8 B^{\frac{2}{a}} }\left( \frac{A^{\frac{2}{a}} a^2}{B^{\frac{2}{a}} b^2}\right)^{\frac{b(a-1)}{a+b}} + \frac{2 A^2 a^2}{N} + 2 B^2 b^2 N^{2b+5}.
\end{align*}
\end{proposition}
Although even the appearance of these estimates is technical, the most important consequence of this result is that it gives an estimate on the convergence rate $\sigma=\sigma_N$ depending on the number of particles in the system.  In particular, we may choose
\begin{align*}
\sigma= \sigma_N=\frac{C}{N^p \vee \rho_K}
\end{align*}
for some $p>0$ where $C>0$ is independent of $N$ and $\rho_K$ is the local Poincar\'{e} constant satisfied by $\mu$ on $K$.  If we believe that the Poincar\'{e} constant $\rho_K$ grows no faster than a polynomial in $N$, then we note that $\sigma$ decays no worse than $C'/ N^{p'}$ as $N\rightarrow \infty$.

\end{example}

\section{Gamma calculus and the case when $W\equiv 1$}
\label{sec:weightone}
In this section, we introduce the basic ideas and methods behind building the weighted $H^1$ distance
\begin{align*}
\| f\|_{\zeta, W} = \sqrt{\int_\mathcal{X} f^2 \, d\mu_W + \int_\mathcal{X} |\nabla_\zeta f|^2 \, d\mu}
\end{align*}
in which the dynamics defined by~\eqref{eqn:main} contracts.

There are two parts to building this norm.  One part determines the weight $W\in L^1(\mu)$ appearing in the measure $\mu_W$ where we recall that $\mu_W$ is defined by
\begin{align*}
\mu_W(d x \, dv) = W(x,v)\, \mu(dx \, dv)
\end{align*}
and $\mu$ is the Boltzmann-Gibbs measure as in~\eqref{eqn:Gibbs}.  The other part of the construction decides how to define the modified gradient operator $\mathcal{\nabla}_{\zeta}$.  While the need for the weight $W$ will become apparent as we allow for general potentials satisfying Assumption~\ref{assump:1} and Assumption~\ref{assump:2}, in this section we will focus solely on the latter part of the construction, in particular on how to define $\nabla_\zeta$.  In the following section, we will build off of the analysis done here and define the weight $W$.

\begin{remark}
Although we have already defined $\nabla_\zeta$ in~\eqref{eqn:modgrad} above, we will leave its meaning open in this section, allowing us to discover it naturally.  At this point at the very least, we know that it needs to have enough structure so that squared distance $\int_\mathcal{X} |\nabla_\zeta f|^2 \, d\mu$ is equivalent to $\int_\mathcal{X} |\nabla f|^2 \, d\mu$ where $\nabla$ is the usual gradient operator on $(\RR^k)^N$.
\end{remark}

\begin{remark}
The results in this section focus on the case when $W\equiv 1$,  the resulting construction outlined here is equivalent to those given in \cite{Villani2009, Baudoin2017a}.  However, our approach more closely follows the methods in~\cite{Baudoin2017a}.
\end{remark}

\subsection{The basic idea: hope for contractivity in $L^2(\mu)$}  The idea behind the construction of the modified gradient $\nabla_\zeta$ becomes clear by starting the analysis hoping that, for some constant $\sigma>0$, the inequality
\begin{align}
\label{eqn:attempest}
\int_\mathcal{X} (P_t f)^2  \, d\mu \leqslant e^{-\sigma t} \int_\mathcal{X} f^2 \, d\mu
\end{align}
is satisfied for all $f\in L^2(\mu)$ with $\int_\mathcal{X} f \, d\mu=0$ and all $t\geqslant 0$.  The bound above would then imply exponential convergence to equilibrium in $L^2(\mu)$ as we have centered the observable $f\in L^2(\mu)$.  Although for the dynamics~\eqref{eqn:main} we cannot arrive at this estimate, by starting here we will see how and why the need for the modified distance $\| f\|_{\zeta, 1}$ arises to combat the absence of a contraction in the usual $L^2(\mu)$ sense.

To begin the analysis and attempt to prove estimates like~\eqref{eqn:attempest}, we need to be able to differentiate with respect to time and commute this operation with the integral in~\eqref{eqn:attempest}.  We will be able to do this employing the appropriate smoothing properties satisfied by the Markov semigroup $\{P_t\}_{t\geqslant 0}$ as outlined in the following lemma and corollary.
\begin{lemma}
\label{lem:hor}
Suppose that the potential $U$ satisfies Assumption~\ref{assump:1}.  Then for all $(x,v) \in \mathcal{X}$, the distribution of the process $(x(t), v(t))$ solving~\eqref{eqn:main} with the initial value $(x(0), v(0))=(x,v)$  is absolutely continuous with respect to the Lebesgue measure on $\mathcal{X}$ and its probability density, denoted by $p_t((x,v), (x', v'))$, is smooth; that is,
\begin{align*}
(t, (x,v), (x', v'))\mapsto p_t((x,v), (x', v')) \in C^\infty((0, \infty) \times \mathcal{X} \times \mathcal{X}).
\end{align*}
\end{lemma}
\begin{proof}
This follows by H\"{o}rmander's hypoellipticity theorem~\cite{Hormander1967a}.  A proof can be found in~\cite{HerzogMattingly2017}.
\end{proof}

A basic corollary of Lemma~\ref{lem:hor} is the following fact which we will use throughout this and the following section.
\begin{corollary}
\label{cor:semireg}
Suppose that the potential $U$ satisfies Assumption~\ref{assump:1}.  For any
 $f\in C^\infty_0(\mathcal{X}; \RR)$, let $\f\in C^\infty_b(\mathcal{X}; \RR)$ be given by
\begin{align}
\label{eqn:centerob}
\f(x,v)= f(x,v)- \int_\mathcal{X} f \, d\mu.
\end{align}
Then $(t, (x,v))\mapsto P_t \f(x,v): [0, \infty) \times \mathcal{X}\rightarrow \RR \in C^\infty_b([s,S] \times \mathcal{X})$ for any $0<s<S< \infty$.  Furthermore, for any $t>0$
\begin{align*}
\frac{d}{dt} P_t \f = P_t L\f = L P_t \f.
\end{align*}
\end{corollary}

Following the idea above, we now try to see if a contraction is possible in $L^2(\mu)$.  To do the calculation, suppose that $f\in C_0^\infty(\mathcal{X}; \RR)$ with $\f=f- \int_\mathcal{X} f\, d\mu$ and differentiate at time $t>0$ employing Corollary~\ref{cor:semireg} to produce
\begin{align*}
\frac{d}{dt}\int_\mathcal{X} (P_t \f)^2  \, d\mu &= 2 \int_\mathcal{X} (P_t \f) L (P_t \f) \, d\mu\\
&=  \int_\mathcal{X} L (P_t \f)^2 - 2 \Gamma(P_t \f)\, d\mu,
\end{align*}
where for $g\in C^2(\mathcal{X}; \RR)$
\begin{align}
\label{eqn:gamma0}
\Gamma(g)& := \tfrac{1}{2}[L g^2 - 2 g L g]= \gamma T |\nabla_v g|^2.
\end{align}
To understand some of the structures in the expression above, since $\mu$ is invariant for $P_t$ we have
$ \int_\mathcal{X} L (P_t \f)^2 \, d\mu =0,$
implying
\begin{align}
\label{eqn:gammafirst}
\frac{d}{dt}\int_\mathcal{X} (P_t \f)^2  \, d\mu &= - 2 \int_\mathcal{X} \Gamma(P_t \f) \, d\mu.
\end{align}
Thus if we were somehow able to produce a Poincar\'{e}-type inequality of the form
\begin{align*}
\int_\mathcal{X} \Gamma(f) \, d\mu  \geqslant \frac{\sigma}{2} \int_\mathcal{X} f^2 \, d\mu
\end{align*}
satisfied for all $f\in H^1(\mu)$ with $\int_\mathcal{X} f \, d\mu=0$ for some constant $\sigma>0$, then we would be done, for then
\begin{align*}
\nonumber
\frac{d}{dt}\int_\mathcal{X} (P_t \f)^2  \, d\mu = - 2 \int_\mathcal{X} \Gamma(P_t \f) \, d\mu  \leqslant -\sigma  \int_\mathcal{X} (P_t \f)^2 \, d\mu.
\end{align*}
Hence the estimate~\eqref{eqn:attempest} would then follow by Gronwall's inequality and a simple approximation argument.
Clearly, however, such an inequality cannot be satisfied for all such centered observables $\bar{f}=f-\int_\mathcal{X} f \, d\mu$ with $f\in C_0^\infty(\mathcal X; \RR)$.  Nevertheless, this idea lays the foundation for what follows.

Central to the issue above are the $x$ directions missing in $\Gamma$, and we need to be able to produce these directions so that we can apply the actual Poincar\'{e} inequality with respect to $\mu$ which reads
\begin{align*}
\int_\mathcal{X} |\nabla_v f|^2 + |\nabla_x f|^2 \, d\mu \geqslant \frac{1}{\rho} \int_\mathcal{X} f^2 \, d\mu
\end{align*}
for all $f\in H^1(\mu)$ with $\int_X f \, d\mu =0$ for some constant $\rho>0$.  In particular, the term $|\nabla_x P_t \f|^2$ is clearly not available in expression~\eqref{eqn:gammafirst}.  Thus to produce the missing directions, one goes back to the very beginning of the idea to modify the form of the original functional, this time differentiating an expression of the form
\begin{align}
\label{eqn:metricP}
\| P_t \f\|_{\zeta, 1}^2= \int_\mathcal{X} (P_t \f)^2 +| \nabla_{\zeta} P_t \f|^2 \, d\mu
\end{align}
where $f\in C_0^\infty(\mathcal{X}; \RR)$ and $\f= f- \int_\mathcal{X} f \, d\mu$.  As opposed to the $L^2(\mu)$ topology, this time we are working with a distance equivalent to the $H^1(\mu)$ norm. From this point, however, the analysis repeats itself as one tries to ``tune" the metric above by choosing the gradient $\nabla_{\zeta}$ appropriately.  We will do this calculation now, but using the methods of the Gamma calculus following the ideas in~\cite{Baudoin2017a}.

\subsection{Gamma calculus and the definition of $\nabla_\zeta$}

For constants $a,b,c \in \RR_{\neq 0}$ to be determined later, define the differential operators $Y$ and $Z$
\begin{align}
Yf= a \nabla_v f  \qquad \text{ and } \qquad Zf  = b \nabla_x f -c  \nabla_v f
\end{align}
acting on $f\in C^1(\mathcal{X};\RR)$.  Also, set
\begin{align}
\label{eqn:Tgrad}
\nabla_\zeta= \zeta^{-1}(Y, Z) \qquad \text{ and } \qquad \mathcal{T}(f)=  |Yf|^2 + |Zf|^2= \zeta^2 |\nabla_\zeta(f)|^2
\end{align}
where $\zeta >0$ will be another tuning parameter.  Next, let
\begin{align}
\Gamma^Y(f,g) = Y f \cdot Y g\qquad \text{ and } \qquad  \Gamma^Z(f,g) = Z f \cdot Z g
\end{align}
with $\Gamma^Y(f):= \Gamma^Y(f,f)$ and $\Gamma^Z(f):= \Gamma^Z(f,f)$.

\begin{remark}
Operationally, the forms $\Gamma^Y$ and $\Gamma^Z$ are the objects that arise naturally by differentiating $\mathcal{T}(P_t \f)$, where $\mathcal{T}$ is as in~\eqref{eqn:Tgrad}, with respect to time.
\end{remark}

In light of the above remarks, let $f\in C_0^\infty(\mathcal{X}; \RR)$ and $\f=f-\int_\mathcal{X} f \, d\mu$.  Applying Corollary~\ref{cor:semireg}, we find that since $\mu$ is invariant for $(P_t)_{t\geqslant 0}$
\begin{align*}
& \frac{d}{dt} \int_\mathcal{X}   \mathcal{T}(P_t \f) \, d\mu
\\
& = \int_\mathcal{X} \frac{d}{dt}( |Y P_t \f|^2 + |Z P_t \f|^2) \, d\mu
\\
&=\int_\mathcal{X} 2 \Gamma^Y(P_t \f, L P_t \f) +2 \Gamma^Z(P_t \f, L P_t \f) \, d\mu  \\
&= \int_\mathcal{X} 2 \Gamma^Y(P_t \f, L P_t \f) - L\Gamma^Y(P_t \f)  + 2 \Gamma^Z(P_t \f, L P_t \f)- L \Gamma^Z(P_t \f) \, d\mu \\
&=: - \int_\mathcal{X} 2\Gamma^Y_2(P_t \f)  +  2\Gamma^Z_2(P_t \f) \, d\mu .
\end{align*}

\begin{remark}
Note that $\Gamma_2^Z$ and $\Gamma_2^Z$ on the last line above are the ``iterates" of the forms $\Gamma^Y$ and $\Gamma^Z$, and they are defined by
\begin{align*}
& \Gamma_2^Y(g) = \frac{1}{2}[L \Gamma^Y(g) - 2\Gamma^Y(g, L g)],
\\
& \Gamma_2^Z(g) = \frac{1}{2}[L \Gamma^Z(g)-2\Gamma^Z(g, Lg)].
\end{align*}
Here by ``iterate" we mean in the sense that $\Gamma(g)$ above is the iterate of the standard product as in~\eqref{eqn:gamma0}.
\end{remark}

We next calculate $\Gamma_2^Y$ and $\Gamma_2^Z$.  Letting $[A,B]:=AB-BA$ denote the commutator of operators $A$ and $B$, for a generic test function $g\in C^1(\mathcal{X}; \RR)$ note that since $[L, Y]= a( \gamma \nabla_v - \nabla_x)$,
\begin{align*}
\Gamma_2^Y(g)& =\gamma T |\nabla_v (Yg)|^2 + Yg \cdot [L, Y] g\\
&=   \gamma T | \nabla_v (Y g)|^2 - \frac{a}{b} Yg \cdot Z g + \big(\gamma - \frac{c}{b}\big) \Gamma^Y(g).
 \end{align*}
Similarly
\begin{align*}
\Gamma_2^Z(g) &= \gamma T | \nabla_v(Zg)|^2+ Zf \cdot [L, Z](g) \\
&=\gamma T|\nabla_v (Zg)|^2 + \frac{c}{b} \Gamma^Z(g)+ \frac{c}{a}\big(\frac{c}{b}-\gamma \big) Yg \cdot Zg+ \frac{b}{a} \nabla^2 U Yg \cdot Zg
 \end{align*}
where the last term  in the last line above, $\nabla^2U$ denotes the Hessian matrix and $Zg$ and $Yg$ are understood to be column vectors.  Combining the two expressions then gives
\begin{align*}
\Gamma_2^Y(g) + \Gamma_2^Z(g) &= \gamma T |\nabla_v Y g|^2 + \gamma T |\nabla_v Z g|^2 + \frac{c}{b} \Gamma^Z(g) + \left( \gamma - \frac{c}{b}\right) \Gamma^Y(g)
\\
&\qquad + \left( \frac{c^2}{ba} -  \frac{c\gamma}{a} - \frac{a}{b}\right) Yg \cdot Zg  + \frac{b}{a} \nabla^2 U Yg \cdot Z g.
\end{align*}
In the above, note that we can assure that the cross term $(\ldots) Yg \cdot Zg$ without the Hessian is zero by picking $a,b,c >0$ such that $a=b=1$ and
\begin{align}
\label{eqn:cchoice}
c =c(\gamma)= \frac{\gamma}{2}+ \sqrt{\frac{\gamma^2}{4} + 1}.
\end{align}
Observe that this choice also implies
\begin{align*}
\gamma - \frac{c}{b}& = \frac{\gamma}{2}\bigg( 1- \sqrt{1+ \frac{4}{\gamma^2}}\bigg) \geqslant - \frac{2}{\gamma}.
\end{align*}
\begin{remark}
As far as we can tell, the choice of $a=b=1$ and $c=c(\gamma)>0$ as above seems optimal for $\gamma \gg 1$, in the sense that it maximizes the coefficient of $\Gamma^Z$ while minimizing the growth in $\gamma$ of the other terms in the expression for $\Gamma_2^Y(g) + \Gamma_2^Z(g)$.
\end{remark}
As a consequence of the calculations above, we record the following result.
\begin{proposition}
Let $a=b=1$ and $c(\gamma)$ be as in~\eqref{eqn:cchoice}.  Then for any $g\in C^2(\mathcal{X}; \RR)$
\begin{align}
&\Gamma_2^Y(g) + \Gamma_2^Z(g) \label{eqn:gamma2eq}
\\
& \geqslant  \gamma T |\nabla_v Y g|^2 + \gamma T |\nabla_v Z g|^2 + \gamma \Gamma^Z(g)- \frac{2}{\gamma} \Gamma^Y(g) + \nabla^2 U Yg \cdot Z g. \notag
\end{align}
 \end{proposition}

We next consider a simple example of a condition on $U$ which allows us to easily conclude exponential convergence to equilibrium in an explicit way from this point in our analysis.
\begin{corollary}
\label{cor:boundedhessian}
Suppose that the potential $U$ satisfies Assumption~\ref{assump:1} and the Hessian $\nabla^2 U$ has globally bounded spectrum; that is, there exists a constant $M>0$ such that
\begin{align}
|\nabla^2 U(x) y | \leqslant M|y|  \,\,\,  \text{ for all }x\in \mathcal{O}, y \in \RR^d.
\end{align}
Suppose also that the Boltzmann-Gibbs measure $\mu$ satisfies the Poincar\'{e} inequality on the whole space $\mathcal{X}$ with constant $\rho >0$.  Define constants
\begin{align*}
\zeta^2 = \frac{2+M^2}{\gamma^2 T} + \frac{c^2}{2T} + \frac{1}{4T} \qquad \text{ and } \qquad \sigma = \frac{\gamma}{4}\min\big\{ 1, \frac{1}{\rho \zeta^2}\big\}
\end{align*}
where $c=c(\gamma)>0$ is as in~\eqref{eqn:cchoice}.  Then we have the following estimate for all $t\geqslant 0$
\begin{align*}
\| P_t f \|_{\zeta, 1}^2 \leqslant e^{-\sigma t} \|f\|^2_{\zeta, 1}
\end{align*}
for any $f\in H^1(\mu)$ with $\int_\mathcal{X} f \, d\mu =0$.

\end{corollary}

\begin{proof}
Observe that by~\eqref{eqn:gamma2eq} and Young's inequality we have
\begin{align*}
\Gamma_2^Y(g) + \Gamma_2^Z(g) \geqslant \frac{\gamma}{2} \Gamma^Z(g) - \frac{2 +M^2}{\gamma} \Gamma^Y(g).
\end{align*}
To apply a standard approximation argument, first let $f\in C_0^\infty(\mathcal{X}; \RR)$ and $\f= f- \int_\mathcal{X} f \, d\mu$.  Corollary~\ref{cor:semireg} then implies for $t>0$

\begin{align*}
& \frac{d}{dt} \int_\mathcal{X} (P_t \f)^2 + |\nabla_{\zeta}(P_t \f)|^2 \, d\mu
\\
& = -2 \int_\mathcal{X} \gamma T \Gamma^Y(P_t \f)+ \frac{1}{\zeta^2}(\Gamma_2^Y(P_t \f)+ \Gamma^Z_2(P_t \f)) \, d\mu \\
&\leqslant - 2\int_\mathcal{X} \Big(\gamma T - \frac{2+M^2}{\gamma\zeta^2}\Big) \Gamma^Y(P_t \f) + \frac{\gamma}{2\zeta^2} \Gamma^Z(P_t \f) \, d\mu \\
&= - \int_\mathcal{X} \Big(\gamma T - \frac{2+M^2}{\gamma \zeta^2}\Big) \Gamma^Y(P_t \f) + \frac{\gamma}{2\zeta^2} \Gamma^Z(P_t \f)\, d\mu \\
&\qquad -\int_\mathcal{X} \Big(\gamma T -\frac{ 2+M^2}{ \gamma \zeta^2}\Big) \Gamma^Y(P_t \f) + \frac{\gamma}{2\zeta^2} \Gamma^Z(P_t \f)\, d\mu \\
& \leqslant - \frac{\gamma}{4} \int_\mathcal{X} |\nabla_\zeta(P_t \f)|^2 \, d\mu - \frac{\gamma}{4\zeta^2} \int_\mathcal{X} |\nabla P_t \f|^2 \, d\mu.
\end{align*}
Since $\mu$ satisfies the Poincar\'{e} inequality on $\mathcal{X}$ with respect to $\mu$ with constant $\rho>0$, we find that
\begin{align*}
\frac{d}{dt} \| P_t \f \|_{\zeta, 1}^2 & \leqslant - \frac{\gamma}{4}\min\bigg\{1 , \frac{1}{ \rho \zeta^2 }  \bigg\}  \| P_t \f\|_{\zeta, 1}
\end{align*}
for $t\geqslant 0$.  This finishes the proof of the result.
\end{proof}

\begin{example}
Returning to the setting of the single-well potential $U(x)= |x|^2/2$ discussed in Example~\ref{ex:single}, we will be able to conclude the convergence claim made there using the result above.  Indeed, since $\nabla^2 U= \text{Id}_{d\times d}$  we have $M= 1$.  Moreover, since $\mu$ is a mean zero Gaussian on $\RR^{2d}$ with covariance matrix $T\text{Id}_{2d\times 2d}$, the Poincar\'{e} constant on $\RR^{2d}$ is $T$; that is, $\rho=T$ where $\rho$ is an in the statement of Corollary~\ref{cor:boundedhessian}.  Consequently, we have the conclusion of Corollary~\ref{cor:boundedhessian} with
\begin{align*}
\zeta^2 = \frac{3}{\gamma^2 T} + \frac{\big(\frac{\gamma}{2}+ \sqrt{\frac{\gamma^2}{4}+1} \big)^2}{2T} + \frac{1}{4T}\qquad \text{ and } \qquad \sigma = \frac{\gamma}{4}\min\{ 1, \frac{1}{T\zeta^2}\}.
\end{align*}
\end{example}

\begin{remark}
In Corollary \ref{cor:boundedhessian}, if one assumes that the Boltzmann-Gibbs measure $\mu$ satisfies the log-Sobolev  inequality on the whole space $\mathcal{X}$, then a similar proof yields a convergence in entropy.  See Section 2.5 in \cite{Baudoin2017a}.
\end{remark}

\subsection{Integrated estimates and Villani's condition}  Clearly, there are many natural potentials  $U$ which do not satisfy the bounded spectrum condition as in the statement of Corollary~\ref{cor:boundedhessian}.  However, as we now see, we can weaken this boundedness assumption by employing integrated estimates and a condition originally due to Villani in \cite{Villani2009}.  Below, we recast this condition in terms of the growth of spectrum of $\nabla^2U$ to obtain better dimensionality dependence on the coefficients defining the estimate.  It should also be noted that, although this condition allows for further growth in $U$, it is not enough to control the type of growth exhibited in singular potentials.
\begin{assump}
\label{assump:3}
There exist constants $\kappa_0 , \kappa_0'>0$ such that
\begin{align}
\label{eqn:Vilcond}
|\nabla^2 U(x) y | \leqslant \kappa_0 |\nabla U(x) | |y | + \kappa_0' |y|
\end{align}
for all $x\in \mathcal{O}$ and all $y\in \RR^d$.  Furthermore, $\kappa_0 >0$ above satisfies
\begin{align*}
\kappa_0 \leqslant \frac{\gamma }{2 \sqrt{T+ Tc^2(\gamma)}}
\end{align*}
where $c(\gamma)>0$ is as in~\eqref{eqn:cchoice}.
\end{assump}

\begin{remark}The upper threshold on $\kappa_0>0$ above is not so important.  Indeed, one can obtain exponential convergence if $\kappa_0>0$ is arbitrary using the methods outlined here.  However, one has to adjust the choices of $a,b, c>0$ accordingly.  To keep what follows simple, we will maintain the original choices and employ the threshold above.
\end{remark}

Note that by Young's inequality, relation~\eqref{eqn:gamma2eq} implies
\begin{align}
& \Gamma_2^Y(g)+ \Gamma_2^Z(g) \label{eqn:gamma3eq}
\\
& \geqslant  \gamma T |\nabla_v Y g|^2 + \gamma T |\nabla_v Z g|^2 + \gamma \Gamma^Z (g)- \frac{2}{\gamma}\Gamma^Y(g) + \nabla^2 U Yg \cdot Zg \notag
\\
&\geqslant  \gamma T |\nabla_v Y g|^2 + \gamma T |\nabla_v Z g|^2 + \frac{\gamma}{2} \Gamma^Z(g) - \mathcal{R}(x, Yg), \notag
\end{align}
where
\begin{align}
\label{eqn:Rdef}
\mathcal{R}(x, y)= \frac{2}{\gamma} |y|^2+ \frac{|\nabla^2 U y |^2}{2\gamma}.
\end{align}
We now note that if the potential satisfies Assumption~\ref{assump:1} and the more mild Assumption~\ref{assump:3}, we can use the next result to subsume the term $\mathcal{R}(x, Yg)$ using integrated estimates.

\begin{remark}
Our argument below closely follows the proof given in Villani's work~\cite{Villani2009}, but it has been adapted to our setting.
\end{remark}

\begin{proposition}
Suppose that the potential satisfies Assumption~\ref{assump:1} and Assumption~\ref{assump:3}.  Then for all $f\in H^2(\mu)$ we have the estimate
\begin{align}
\nonumber \frac{\gamma}{4Tc^2} \int_\mathcal{X} |\nabla U|^2 \Gamma^Y(f) \, d\mu &\leqslant \int_\mathcal{X} \bigg(\frac{2 \gamma T \kappa_0 d}{4c^2 } + \frac{\sqrt{2d} \kappa_0' \gamma}{4c^2} \bigg)  \Gamma^Y(f)  \, d\mu \\ \qquad &\qquad + \int_\mathcal{X} \gamma T  |\nabla_v (Zf)|^2+ \gamma T  |\nabla_v (Yf)|^2 \, d\mu
\end{align}
where $c(\gamma)>0$ is as in~\eqref{eqn:cchoice}. \end{proposition}

\begin{proof}
Observe that for $f\in C_0^\infty(\mathcal{X}; \RR)$, integration by parts implies
\begin{align*}
& \frac{1}{T}\int_\mathcal{O} |\nabla U|^2 \Gamma^Y(f) e^{-\frac{U}{T}} \, dx = - \int  \nabla U \Gamma^Y(f) \cdot \nabla( e^{-\frac{U}{T}}) \, dx \\
&= \int_\mathcal{O}  \Delta U \Gamma^Y(f) e^{-\frac{U}{T}} \, dx + 2 \sum_{j, \ell} \int_\mathcal{O} \partial_{x_j} U \partial_{v_\ell} f \partial^2_{v_\ell x_j} f e^{-\frac{U}{T}} \, dx \\
&\leqslant \int_\mathcal{O} | \Delta U|  \Gamma^Y(f) e^{-\frac{U}{T}} \, dx + 2 \int_\mathcal{O} |\nabla U| |Y f| |Y( \nabla_x f)| e^{-\frac{U}{T}} \, dx.
\end{align*}
Writing the expressions in terms of $Y$ and $Z$ produces for any $\alpha >0$
\begin{align*}
& \alpha \int_\mathcal{O} |\nabla U|^2 \Gamma^Y(f) e^{-\frac{U}{T}} \, dx
\\
& \leqslant
\alpha T \int_\mathcal{O} |\Delta U| \Gamma^Y(f) e^{-\frac{U}{T}} \, dx + 2\alpha T  \int_\mathcal{O} |\nabla U| |Y f| |\nabla_v (Zf)| e^{-\frac{U}{T}} \,dx \\
 &\qquad  +2 c \alpha T \int_\mathcal{O} |\nabla U| |Y f| |\nabla_v (Yf)| e^{-\frac{U}{T}}\, dx.
\end{align*}
Observe that if $e_i$ denotes the standard orthonormal basis on $(\RR^k)^N$, Assumption~\ref{assump:3} implies
\begin{align*}
|\Delta U |= \sqrt{\sum_{i} |\nabla^2 U e_i \cdot e_i|^2 } \leqslant \sqrt{2d}\kappa_0 |\nabla U| + \sqrt{2d} \kappa_0'.
\end{align*}
Plugging this into the estimate above, one can then apply Young's inequality several times to arrive at the claimed estimate by picking $\alpha =\gamma/(4Tc^2)$
\end{proof}

Applying the previous result, we arrive at the needed integrated estimate in this special case.
\begin{corollary}
\label{cor:VilH2}
Suppose that the potential $U$ satisfies Assumption~\ref{assump:1} and Assumption~\ref{assump:3}.  Then for all  $g\in C^\infty_b(\mathcal{X};\RR)$ we have
\begin{align*}
\int_\mathcal{X} \Gamma_2^Y(g) + \Gamma_2^Z(g) \, d\mu \geqslant \int_\mathcal{X} \frac{\gamma}{2} \Gamma^Z(g) -\Big( \frac{2+ M^2}{\gamma}\Big)  \Gamma^Y(g) \, d\mu
\end{align*}
where
\begin{align}
\label{eqn:M22}
M^2=\frac{2\gamma^2 T \kappa_0 d}{4 c^2(\gamma)} + \frac{\sqrt{2d} \kappa_0' \gamma^2}{4 c^2(\gamma)} + (\kappa_0')^2.
\end{align}
Consequently, the conclusion of Corollary~\ref{cor:boundedhessian} holds with this choice of $M$.
\end{corollary}

\begin{remark}
If $U$ satisfies Assumption~\ref{assump:1} and Assumption~\ref{assump:3}, then from \cite[Appendix]{Villani2009},  for some constant $C>0$ we have the following regularization estimate for all $t\in  (0,1]$
\begin{align}\label{reverse poincare}
\| P_t f \|_{\zeta, 1}^2 \leqslant \frac{C}{t^3} \|f\|_{L^2(\mu)}^2
\end{align}
for any $f\in L^2(\mu)$ with $\int_\mathcal{X} f \, d\mu =0$.  Combining this small-time estimate with the conclusion of  Corollary \ref{cor:VilH2} yields a convergence in $L^2(\mu)$
\begin{align}\label{convergence L2}
\| P_t f \|^2_{L^2(\mu)} \leqslant C e^{-\sigma t}  \|f\|^2_{L^2(\mu)}
\end{align}
where  $f\in L^2(\mu)$ with $\int_\mathcal{X} f \, d\mu =0$, $C>0$ is a constant and $\sigma >0$ is explicit. Indeed, the conclusion of  Corollary \ref{cor:VilH2} and then  \eqref{reverse poincare} yield
\[
\| P_{t+1} f \|_{\zeta, 1}^2 \leqslant e^{-\sigma t} \| P_1 f \|_{\zeta,1} \leqslant C e^{-\sigma t} \|f\|_{L^2(\mu)}^2.
\]
Thus, for $t \ge 1$,
\[
\| P_t f \|^2_{L^2(\mu)} \leqslant Ce^{\sigma} e^{-\sigma t}  \|f\|^2_{L^2(\mu)},
\]
and since  for $t \in [0,1]$,
\[
\| P_t f \|^2_{L^2(\mu)} \leqslant  \|f\|^2_{L^2(\mu)}
\]
one concludes \eqref{convergence L2}.
\end{remark}

\begin{example}
A typical example of a potential which does not have a bounded spectrum but satisfies Villani's condition \eqref{assump:3} is the double-well potential $U: (\RR^k)^N\rightarrow [0, \infty)$ given by $U(x)= (|x|^2-1)^2/4.$  Here, we will apply Corollary~\ref{cor:VilH2} to obtain the estimates on the convergence rate given in Example~\ref{ex:double}.

Clearly $U$ satisfies Assumption~\ref{assump:1}.  Also note that
\begin{align*}
|\nabla U(x)| = |x| | |x|^2 -1| \qquad \text{ and } \qquad \nabla^2 U (x)= (|x|^2 -1) \text{Id}_{d\times d} + A(x)
\end{align*}
where $(A(x))_{ij}= 2 x_i x_j$.  Let
\begin{align*}
\kappa_0 = \frac{\gamma }{2\sqrt{T+Tc^2(\gamma)}}
\end{align*}
where $c(\gamma)= \frac{\gamma}{2}+ \sqrt{\frac{\gamma^2}{4}+1}$ is as in~\eqref{def:cgamma}.
It then follows that for any $y\in\RR^d$
\begin{align*}
|\nabla^2 U(x) y | &\leqslant ||x^2|-1| |y| + |A(x)| |y| = (||x|^2 -1| + 2|x|^2 ) |y| \\
&\leq( 3||x|^2 -1| + 2) |y|\\
& \leqslant  \kappa_0 |\nabla U ||y| + \Big(\frac{27}{\kappa_0^2}+ 2\Big) |y|:= \kappa_0 |\nabla U||y| + \kappa_0' |y|.
\end{align*}
Thus $U$ satisfies Assumption~\ref{assump:3}.  One can check that $U$ satisfies a Poincar\'{e} inequality for some constant $\rho >0$ and that $M^2$ can be chosen as in Corollary~\ref{cor:VilH2}.
Thus we obtain the result of Corollary~\ref{cor:VilH2} for this potential with
\begin{align*}
& \zeta^2 = \frac{2+M^2}{\gamma^2 T} + \frac{\big(\frac{\gamma}{2}+ \sqrt{\frac{\gamma^2}{4}+1} \big)^2}{2T} + \frac{1}{4T},
\\
& \sigma = \frac{\gamma}{4}\min\{ 1, 1/\rho\zeta^2\}.
\end{align*}

\end{example}

Now that we see the utility in the Gamma calculus even in settings where $\nabla^2U$ does not have bounded spectrum, we next turn to the weighted setting which will allow us to deal with potentials which do not satisfy Assumption~\ref{assump:3}.

\section{The Weighted Setting and Reduction of the Main Result to The Existence of a Lyapunov function}
\label{sec:weight}
In this section, we adapt the Gamma calculus discussed in the previous section to work in settings where the Hessian $\nabla^2 U$ grows faster relative to the gradient than in the Villani-type condition outlined in Assumption~\ref{assump:3}.  The main ingredient needed to get around the issue of growth is the weight $W\in L^1(\mu) \cap C^2(\mathcal{X};[1, \infty))$ appearing in the distance $\| \cdot \|_{\zeta, W}$.  Given that the weight $W$ satisfies a certain Lyapunov structure, we can obtain explicit estimates on the exponential rate of convergence to equilibrium.  In the next section, we will see that this Lyapunov structure is not too hard to produce by exhibiting the appropriate functional, thus allowing us to conclude Theorem~\ref{thm:mainconv}.

\begin{remark}
The use of a weighted measure to establish convergence to equilibrium for Langevin dynamics in the setting of ``polynomial-like" potentials was done originally by Talay in~\cite{Talay2002}.  There, the weight is defined with respect to Lebesgue measure on the phase space and the Lyapunov condition is with respect to the formal $L^2(dxdv)$-adjoint $L^\dagger$ of $L$.  Here, by considering weights with respect to the invariant measure $\mu$, we produce a condition more reminiscent of the usual, total-variation Lyapunov condition.  Indeed, this is because the formal adjoint $L^*$ of $L$ with respect to $L^2(\mu)$ leads to the almost the same dynamics~\eqref{eqn:main}, except that the Hamiltonian part of the operator $L$ has been ``time-reversed".  See the top of p. 23 for further details.
\end{remark}

\begin{remark}
It should be noted that the choice of a Lyapunov functional in the following section may not at all be optimal.  Thus the general result given in this section provides a way to ``fine tune" the construction of the weight $W$ to produce better convergence bounds.
\end{remark}

\begin{remark}
To get around the Villani-type growth condition, a weighted measure approach was taken in the interesting recent paper~\cite{CGMZ_2017}, but with the weight $W$ appearing next to the term $|\nabla_x f|^2$.  While it is certainly possible to follow this approach to obtain exponential convergence to equilibrium in an explicitly measurable way, in~\cite{CGMZ_2017} the Villani-type growth condition is replaced by another growth condition on $U$ that is not satisfied by common potentials with singularities, such as the Lennard-Jones potential for example.  To see their condition, we refer the reader to Corollary~3 of \cite{CGMZ_2017}.  The reason why these conditions are not satisfied is that they assume the global bound, for some $\eta >0$
\begin{align}
\label{eqn:guilcond}
|\nabla U|^2 \geqslant c_1 U^{2\eta +1} \geqslant c_2 U^{2\eta} \geqslant c_3 |\nabla^2 U|
\end{align}
where the $c_i >0$ are constants.  A simple example of where this condition is violated is the one-dimensional potential $U:\RR\rightarrow (0, \infty]$ defined by
\begin{align*}
U(x)= \begin{cases}
x^a + \frac{1}{x^b}& \text{ if } x>0\\
\infty & \text{ if } x\leqslant 0
\end{cases}\end{align*}
where $a>1, b >0$ are constants.  To see why relation~\eqref{eqn:guilcond} fails, near $x=0$ one must choose $\eta =\tfrac{\beta}{2}+1$ but then the bound $|\nabla U|^2 \geqslant c_1 U^{2\eta +1}$ is no longer satisfied for $x\gg1$.  In essence, these conditions fail because there are two different scaling regimes for $U$ and its derivatives: one near $0$ and one near $\infty$.  The same reasoning can be used to show that the Lennard-Jones potential, or Coloumb potential for that matter, in general dimensions does not satisfy~\eqref{eqn:guilcond}.
\end{remark}

\begin{remark}
In another paper~\cite{BakryCattiauxGuillin2008}, the weighted approach like the one here is adopted and compared with another approaches, such as convergence in the total variation distance, but in the fully elliptic setting of gradient systems where $L$ is symmetric with respect to the $L^2(\mu)$ inner product.  The approach outlined here works in the weakly hypoelliptic setting of Langevin dynamics where $L$ no longer has this symmetry.
\end{remark}

 As before, letting $f\in C_0^\infty(\mathcal{X}; \RR)$ and $\f=f - \int_\mathcal{X} f \, d\mu$, we could start our analysis by differentiating the weighted $L^2$ functional
\begin{align*}
\int_X (P_t \f)^2 d\mu_W.
\end{align*}
However, we would wind up in the same situation in the setting where $W\equiv 1$, missing the $x$ directions in the gradient.  Thus we start from the point of checking that the time derivative of
\begin{align*}
\| P_t \f \|_{\zeta, W}^2 = \int_X (P_t \f)^2 \, d\mu_W + \int_X |\nabla_{\zeta}(P_t \f)|^2 \, d\mu
\end{align*}
does precisely what it was constructed to do.  One difference from the non weighted setting is that we will have to deal with a term that was  zero in the case when $W\equiv 1$ by invariance of $\mu$.  As we will see, a certain Lyapunov structure, which ultimately corresponds to the weight $W$, will allow us to control such a term.

As in the case when we assumed that the Hessian $\nabla^2 U$ was bounded, we will have to rescale in order to define the gradient $\nabla_{\zeta}$.  In particular, we again let
\begin{align}
\nabla_{\zeta}=\zeta^{-1} (Y, Z)
\end{align}
for some $\zeta >0$ to be determined later.  In what follows, though, we maintain the choices of the coefficients of the operators $Y$ and $Z$:
\begin{align*}
a= b=1,\qquad c=\frac{\gamma}{2}+ \sqrt{\frac{\gamma^2}{2}+1}.
\end{align*}

Recalling the notation
\begin{align*}
\mathcal{T}(g) = \Gamma^Y(g)+ \Gamma^Z(g) =  \zeta^{2}|\nabla_{\zeta}(g)|^2,
\end{align*}
we define
\begin{align}
\label{eqn:T2def}
\mathcal{T}_2(g)= \Gamma_2^Y(g)+ \Gamma_2^Z(g).
\end{align}
Let $f\in C_0^\infty(\mathcal{X}; \RR)$ with $\f=f- \int_\mathcal{X} f \, d\mu$.   Employing Corollary~\ref{cor:semireg} again and the fact that $W\in L^1(\mu)$, we have for $t>0$
\begin{align*}
& \frac{d}{dt} \|P_t \f\|_{\zeta, W}^2
\\
&= \int_\mathcal{X}  L(P_t \f)^2 -   2\gamma T \Gamma^Y(P_t \f) \, d\mu_W +  \frac{1}{\zeta^2} \int_\mathcal{X} L\mathcal{T}(P_t \f) -2 \mathcal{T}_2(P_t \f) \, d\mu\\
&= \int_\mathcal{X} L (P_t \f)^2   - 2\gamma T \Gamma^Y(P_t \f)  \, d\mu_W - \frac{1}{\zeta^2} \int_\mathcal{X} 2 \mathcal{T}_2 (P_t \f)  \, d\mu
\end{align*}
where we used the fact that $\mu$ is invariant for $(P_t)_{t\geqslant 0}$ on the last line above.
To control these terms, note that if we could show the Poincar\'{e}-type bound
\begin{align}
\label{eqn:PIlyap}
\| g\|_{\zeta,  W}^2 \leqslant \frac{1}{\sigma}\bigg[ \int_\mathcal{X} 2\gamma T \Gamma^Y(g)  - L g^2   \, d\mu_W +   \int_\mathcal{X}  \frac{2}{\zeta^2} \mathcal{T}_2(g)  d\mu \bigg]
\end{align}
for some constant $\sigma>0$ for all $g\in C^\infty_b(\mathcal{X}; \RR)$ with $\int_X g\, d\mu =0$, then we would be finished.  In particular, we have the following result.

\begin{lemma}
Suppose that there exists $W\in L^1(\mu) \cap C^2(\mathcal{X}; [1, \infty))$ and constants $\sigma, \zeta>0$ such that the estimate~\eqref{eqn:PIlyap} holds for all $g\in C^\infty_b(\mathcal{X}; \RR)$ with $\int g\, d\mu =0$.  Then for all $t\geqslant 0$, we have the bound
\begin{align}
\|P_t f \|_{\zeta, W}^2 \leqslant e^{-\sigma t} \| f\|_{\zeta, W}^2
\end{align}
for all $f\in H^1_{\zeta, W}$ with $\int_\mathcal{X} f \, d\mu =0$.
\end{lemma}

In order to prove an inequality of the form~\eqref{eqn:PIlyap}, we will need to use a certain Foster-Lyapunov structure coupled with a local Poincar\'{e} inequality as in~\eqref{def:LPI}.  While the Lyapunov structure provides a means by which to control excursions far from the ``center" of space in $\mathcal{X}$, the local Poincar\'{e} inequality provides the mechanism for mixing and/or coupling when two processes started from different initial conditions return to this center.  Different from the typical Harris Theorem for Markov chains~\cite{HairerMattingly2008, MeynTweedie1993a}, however, the Lyapunov condition will look slightly strange because it will be with respect to the $L^2(\mu)$-adjoint $L^*$ of the generator $L$.  In other words, $L^*$ is defined by the rule
\begin{align*}
\int L^* f \, g \, d\mu = \int f \, Lg \, d\mu \,\, \text{ for all } \,\,  f, g \in C_0^\infty(\mathcal{X}; \RR).
\end{align*}
By a straightforward calculation, since $\mu$ has the Boltzmann-Gibbs form as in \eqref{eqn:Gibbs}, the operator $L^*$ takes the form
\begin{align}
\label{def:Lstar}
L^*= -v\cdot \nabla_x - \gamma v\cdot \nabla_v + \nabla U \cdot \nabla_v + \gamma T \Delta_v .
\end{align}
Observe that the only difference between $L^*$ and $L$ is that the Hamiltonian part
\begin{align*}
v\cdot \nabla_x-\nabla U \cdot \nabla_v
\end{align*}
of $L$ has been time-reversed in $L^*$.  At the level of the Foster-Lyapunov criteria, we will see that this means that the fundamental structures giving rise to a Lyapunov function in $L$ easily translate to the existence of a Lyapunov function for $L^*$.  Intuitively, dissipation still acts on the velocity directions $v$ through $- \gamma v \cdot \nabla_v$.  This quantity is then averaged along the deterministic Hamiltonian dynamics, but this time in the opposite direction along $L^*$.


\begin{definition}
\label{def:Lyap}
We say that a function $W\in C^2(\mathcal{X}; [1, \infty))$ is a \emph{strong Lyapunov function with respect to} $L^*$ \emph{with constants $\alpha, \beta >0$ and set $J\subseteq \mathcal{X}$} if
\begin{itemize}
\item[(i)]  $W\rightarrow \infty$ as $H\rightarrow \infty$;
\item[(ii)] $J$ is compact and connected and $W$ satisfies the global bound
\begin{align*}
L^* W \leqslant - \alpha W + \beta \mathbf{1}_{J} .
\end{align*}
\item[(iii)] For all $g\in C^\infty_b(\mathcal{X}; \RR)$ we have $Lg \, W \in L^1(\mu)$,  $g \, L^*W \in L^1(\mu)$ and
\begin{align*}
\int_\mathcal{X} L g \,  W \, d\mu = \int_\mathcal{X} g \, L^* W \, d\mu .
\end{align*}\end{itemize}
We say that a function $W\in C^2(\mathcal{X}; \RR)$ is a \emph{weak Lyapunov function with respect to} $L^*$ \emph{with constants $\alpha, \beta >0$ and set $J$} if $W$ satisfies properties (i), (ii), and (iii).
\end{definition}

\begin{remark}
Note that the only difference between a weak Lyapunov function and strong Lyapunov function is the range of $W$.  From a weak Lyapunov function, by property (i) one can easily produce a strong Lyapunov function by adding a large enough constant.
\end{remark}

We have the following simple consequence of the definition above.  Note that the result follows in a different manner once one realizes that $\mu$ is also invariant for the diffusion process with generator $L^*$.

\begin{proposition}
Suppose that $W\in C^2(\mathcal{X}; (0, \infty))$ is a weak Lyapunov function with respect to $L^*$ with constants $\alpha, \beta >0$ and set $J\subseteq \mathcal{X}$.  Then
\begin{align*}
\int_\mathcal{X} W \, d\mu \leqslant \frac{\beta}{\alpha} \mu(J).
\end{align*}
In particular, $W\in L^1(\mu)$.
\end{proposition}

\begin{proof}
Applying Definition~\ref{def:Lyap} (iii) by plugging in $g=1$ we have
\begin{align*}
0 = \int_\mathcal{X} L(1) \, W \, d\mu = \int_\mathcal{X} L^*W \, d\mu.
\end{align*}
Hence, applying Definition~\ref{def:Lyap} (ii) we obtain
\begin{align*}
0\leqslant \int_\mathcal{X} -\alpha W + \beta \mathbf{1}_J \, d\mu.
\end{align*}
Rearranging the above finishes the proof.
\end{proof}

We next state and prove our main theoretical tool.  Supposing $U$ satisfies Assumption~\ref{assump:1} and Assumption~\ref{assump:2}, in essence it states that given the existence of a Lyapunov function with respect to $L^*$ with constants $\alpha, \beta >0$ and set $J$ then, provided the Lyapunov function satisfies an additional but nominal growth condition, we can arrive at the Poincar\'{e}-type inequality of the form~\eqref{eqn:PIlyap}.  As we will see later, these hypotheses are completely redundant.  That is, under Assumption~\ref{assump:1} and Assumption~\ref{assump:2}, we can always exhibit an explicit Lyapunov function with respect to $L^*$ satisfying these growth conditions.  However, we keep the statement as is, emphasizing that if one wants optimal rates of convergence, there may be room for improvement by constructing a ``better" functional.

\begin{theorem}
\label{thm:poincarelyap}
Suppose that the potential $U$ satisfies Assumption~\ref{assump:1} and Assumption~\ref{assump:2}, and that there exists a strong Lyapunov function $V$ for $L^*$ with constants $\alpha, \beta>0$ and set $J\subseteq \mathcal{X}$.  Let $\rho>0$ denote the local Poincar\'{e} constant determined by the measure $\mu$ on $J$.  Define $\rho'>0$ by
\begin{align*}
\rho' = \frac{(4 c^2(\gamma) + 4)\rho}{\gamma}
\end{align*}
where $c=c(\gamma)>0$ is as in~\eqref{eqn:cchoice}.
Suppose there exists a constant $\lambda >0 $ for which $W=V+\lambda$ and $V$ respectively satisfy
\begin{align}
\label{eqn:Wchoice}
& W(x,v) |y|^2\geqslant  (\beta \rho'+1) \left(\frac{\mathcal{R}(x, y)}{\gamma T} + \frac{|y|^2}{2T}\right),
\\\label{eqn:Vchoice}
& V(x,v)\geqslant \frac{2\beta \mu(J^c)}{\alpha \mu(J)},
\end{align}
for all $(x,v) \in \mathcal{X}, y\in \RR^d$ where $\mathcal{R}$ is as in~\eqref{eqn:Rdef}.  Then the estimate~\eqref{eqn:PIlyap} holds with this choice of $W$ and for $\zeta, \sigma >0$ given by
\begin{align*}
\zeta^2 =\frac{2}{1+ \rho' \beta} \qquad \text{ and } \qquad \sigma= \frac{\alpha}{2(1+\lambda)} \wedge  \frac{\gamma }{1+ \rho' \beta}.\end{align*}

\end{theorem}
\begin{proof}[Proof of Theorem~\ref{thm:poincarelyap}]
We begin the proof of the result by making a few basic observations.  Recall that by~\eqref{eqn:gamma3eq}, we have the following integrated estimate
\begin{align}
\label{eqn:Lgammabound}
\int \mathcal{T}_2(f) \, d\mu = \int \Gamma^Y_2(f) + \Gamma^Z_2(f) \, d\mu  \geqslant  \int \frac{\gamma}{2} \Gamma^Z (f) - \mathcal{R}(x, Yf) \, d\mu
\end{align}
for any $f\in C^\infty_b(\mathcal{X}; \RR)$.
Also, if $\mu$ satisfies the local Poincar\'{e} inequality on $J$ with constant $\rho>0$, then for all $f\in H^1(\mu)$ we have
\begin{align*}
\int_J f^2 \, d\mu &\leqslant \rho \int_J |\nabla f|^2 \, d\mu +\frac{1}{\mu(J)} \big( \int_J f \, d\mu\big)^2\\
&\leqslant \rho' \int_J \frac{\gamma}{2} \mathcal{T}(f) \, d\mu + \frac{1}{\mu(J)}\big( \int_J f \, d\mu \big)^2 .
\end{align*}
Combining the previous two estimates we obtain, for any $f\in C_b^\infty(\mathcal{X}; \RR)$,
\begin{align}
 \int_J f^2 \, d\mu&\leqslant \rho' \int_\mathcal{X} \mathcal{R}(x, Yf)  + \frac{\gamma}{2} \mathcal{T}(f)-\mathcal{R}(x, Yf)  \, d\mu + \frac{1}{\mu(J)}\left( \int_J f \, d\mu\right)^2
 \notag \\
& \leqslant \rho' \int_\mathcal{X} \left(\mathcal{R}(x, Yf)+ \frac{\gamma}{2} \Gamma^Y(f) \right)\, d\mu + \rho'  \int_\mathcal{X} \mathcal{T}_2(f) \, d\mu + \frac{1}{\mu(J)}\left( \int_J f \, d\mu \right)^2. \label{eqn:res_bound}
\end{align}
Now, let $\lambda >0$ satisfy the hypotheses of the statement and recall that $W= V + \lambda\leqslant V(1+ \lambda) $ since $V\geqslant 1$.  Let $g\in C^\infty_b(\mathcal{X}; \RR)$ with $\int_\mathcal{X} g \, d\mu =0$ and note that rewriting the bound in Definition~\ref{def:Lyap} (ii) produces the following inequality
\begin{align}
\label{eqn:subtract}
 \int_\mathcal{X} g^2 V  \, d\mu \leqslant  \int_\mathcal{X} g^2 \, d\mu_W &\leqslant (1+ \lambda) \int_\mathcal{X} g^2 V \, d\mu \\
\nonumber & \leqslant \frac{1+\lambda}{\alpha}\bigg[ - \int_\mathcal{X} g^2 L^* W \, d\mu + \int_J \beta g^2 \, d\mu \bigg].
 \end{align}
Applying the inequality~\eqref{eqn:res_bound} using the fact that $\int_J g \, d\mu = \int_{J^c} g \, d\mu$,
we obtain
 \begin{align*}
&\beta\int_J g^2 \, d\mu\\
 &\leqslant \beta \rho'  \int_\mathcal{X}\Big( \mathcal{R}(x, Yg) + \frac{\gamma}{2}\Gamma^Y(g)\Big)   \, d\mu + \beta \rho' \int_\mathcal{X} \mathcal{T}_2(g) \, d\mu   + \frac{\beta}{\mu(J)}\bigg( \int_J g \, d\mu\bigg)^2\\
& \leqslant \beta \rho'  \int_\mathcal{X} \Big( \mathcal{R}(x, Yg)+ \frac{\gamma}{2}\Gamma^Y(g)\Big) \, d\mu+ \beta \rho'\int_\mathcal{X} \mathcal{T}_2(g)\, d\mu+ \beta \frac{\mu(J^c)}{\mu(J)}\int_\mathcal{X} g^2 \, d\mu  \end{align*}
where in the last inequality we applied Jensen's inequality to $(\int_{J^c} g \, d\mu )^2$.  Note by~\eqref{eqn:Wchoice} and~\eqref{eqn:Vchoice}, we then arrive at
\begin{align*}
\beta \int_J g^2 \, d\mu &\leqslant \int_\mathcal{X} \gamma T \Gamma^Y(g)  \, d\mu_W + \beta \rho'  \int_\mathcal{X} \mathcal{T}_2(g) \, d\mu + \frac{\alpha}{2} \int_\mathcal{X} g^2 \, d\mu_V.
\end{align*}
Combining this estimate with~\eqref{eqn:subtract}, we find that
\begin{align*}
\frac{1}{2}\int_\mathcal{X} g^2 \, d\mu_W \leqslant \frac{1+\lambda}{2}\int_\mathcal{X} g^2 \, d\mu_V & \leqslant \frac{1+\lambda}{\alpha} \bigg[ - \int_\mathcal{X} g^2 L^*W \, d\mu + \gamma T \int_\mathcal{X} \Gamma^Y(g) \, d\mu_W  \bigg] \\
&\qquad +\frac{(1+\lambda) \beta \rho'}{\alpha} \int_\mathcal{X} \mathcal{T}_2(g) \, d\mu .
\end{align*}
Next note that we can argue similarly as before to see that
\begin{align*}
\int_\mathcal{X} |\nabla_\zeta g|^2 \, d\mu &= \frac{2}{\zeta^2 \gamma}\int_\mathcal{X} \frac{\gamma}{2}\mathcal{T}(g) \, d\mu \leqslant \frac{2}{\zeta^2 \gamma}\bigg[\int_\mathcal{X} \gamma T  \Gamma^Y(g)  \, d\mu_W + \int_\mathcal{X}  \mathcal{T}_2(g) \, d\mu\bigg].
\end{align*}
Putting the previous two estimates together and letting $\delta >0$ be a parameter to be determined shortly, we find that
\begin{align*}
\frac{1}{2}\int_\mathcal{X} g^2 \, d\mu_W + \delta \int_\mathcal{X} |\nabla_\zeta g|^2 \, d\mu & \leqslant  \frac{1+\lambda}{\alpha}\bigg[ - \int_\mathcal{X} g^2 L^* W \, d\mu + \gamma T \int_\mathcal{X} \Gamma^Y(g) \, d\mu_W  \bigg] \\
&\qquad +\frac{(1+\lambda) \beta \rho'}{\alpha} \int_\mathcal{X} \mathcal{T}_2(g) \, d\mu  \\
&\qquad + \frac{2 \delta}{\zeta^2 \gamma} \bigg[ \int_\mathcal{X} \gamma T \Gamma^Y(g) \, d\mu_W +  \int_\mathcal{X} \mathcal{T}_2(g) \, d\mu \bigg].
\end{align*}
Picking
\begin{align*}
\delta = \frac{\gamma \zeta^2(1+\lambda)}{2\alpha} \qquad \text{ and } \qquad \zeta^2 = \frac{2}{1+ \beta \rho'},
\end{align*}
we arrive at the estimate
\begin{align*}
&\frac{1}{2}\int_\mathcal{X} g^2 \, d\mu_W + \delta \int_\mathcal{X} |\nabla_\zeta g|^2 \, d\mu \\
& \leqslant \frac{1+\lambda}{\alpha}\bigg[ - \int_\mathcal{X} g^2 L^* W \, d\mu + 2 \int_\mathcal{X} \gamma T \Gamma^Y(g) \, d\mu_W \bigg]  + \frac{1+\lambda}{\alpha} \int_\mathcal{X} \frac{2}{\zeta^2} \mathcal{T}_2(g) \, d\mu.
 \end{align*}
 The result now follows by bounding the quantity on the left-hand side below by $\min\{ 1/2, \delta \} \|g\|_{\zeta, W}^2$.
\end{proof}

\section{The existence and consequences of quantitative Lyapunov functionals}
\label{sec:LFs}

In this section, we prove that if the potential $U$ satisfies Assumptions~\ref{assump:1} and Assumption~\ref{assump:2}, then there exists a Lyapunov function $V$ for $L^*$ whose constants $\alpha, \beta>0$ and set $K$ can be explicitly estimated.  Once we establish the existence of such a function $V$, we will, with an additional auxiliary estimate, plug it into Theorem~\ref{thm:poincarelyap} to conclude the main convergence result, namely Theorem~\ref{thm:mainconv}.

In order to construct the desired Lyapunov function, we follow the approach in \cite{HerzogMattingly2017} by exponentiating the Hamiltonian with the appropriately chosen correction $\psi$.  Recall that we need to find a Lyapunov function with respect to $L^*$ as in~\eqref{def:Lstar} as opposed to $L$.  This difference, however, only results in placing a minus sign in front of the correction.

\begin{theorem}
\label{thm:LFquant}
Suppose that $U$ satisfies Assumption~\ref{assump:1} and Assumption~\ref{assump:2}.  Consider the constants $R_2>R_1>0$ given in~\eqref{eqn:constr1r2}, let $b\in (0,\frac{1}{2Td}]$ and define $h\in C^\infty([0,\infty); [0,1])$ to be any function satisfying
\begin{align*}
h(q)= \begin{cases}
1 & \text{ if } q\geqslant R_2 \\
0 & \text{ if } q\leqslant R_1
\end{cases}
\,\,\, \text{ and } \,\,\,
|h'| \leqslant \frac{2}{R_2 - R_1}= \frac{1}{16 Td }.
\end{align*}
Let $\psi \in C^\infty(\mathcal{X}; \RR)$ be given by
\begin{align*}
\psi(x,v) =\begin{cases}
 \displaystyle{-\frac{3}{2}\gamma b T  d\,    \frac{ h(U(x))v\cdot \nabla U(x)}{|\nabla U(x)|^2} }& \text{ if } U(x) \geqslant R_1 \\
 0 & \text{ otherwise}
 \end{cases}.
\end{align*}
Then
\begin{align}
\label{eqn:Vlyap1}
V(x,v) = \exp\bigg(b H(x,v) + \psi(x,v)\bigg)
\end{align}
is a weak Lyapunov function corresponding to $L^*$ with constants
\begin{align*}
\alpha = \frac{\gamma b T d}{4}, \qquad \beta = \frac{5 \gamma b T d}{4} e^{ b R_2 + 5 bT d}
\end{align*}
and  set
\begin{align}
\label{eqn:compact_set_def}
K= \{ |v|^2 \leqslant (20e^4+2) Td \} \cap \{ U \leqslant R_2 \}.
\end{align}
Moreover, $V\geqslant e^{-1}$ so that $\tilde{V}:=e^{1} V$ is a strong Lyapunov function corresponding to $L^*$ with constants $\alpha, e^{1} \beta$ and set $K$ as above.

\end{theorem}

\begin{remark}
Note that the definition of the constants $\alpha, \beta >0$ in the result above suggests choosing $b\propto 1/Td$.  Thus in terms of the dimension $d$, the important part will be to determine the precise dependence of $R_2$ on $d$.  We already know that $R_2$ is at least on the same order as $Td$ since $R_2=R_1 + 32 Td$.
\end{remark}

%

\begin{proof}[Proof of Theorem~\ref{thm:LFquant}]
We recall that by the choice of $R_2$, $K$ is both compact and connected (cf. Remark~\ref{rem:connected}).
Let $b\in (0,\frac1{2Td}]$ be as in the statement of the result and let $\delta >0$ be a constant which we will choose momentarily.  Consider a candidate Lyapunov functional $V_{b, \delta}: \mathcal{X}\rightarrow (0, \infty)$ defined by
\begin{align}
V_{b, \delta}(x,v) = \exp( b H(x,v) + \psi_\delta(x,v))
\end{align}
where $\psi_\delta \in C^\infty(\mathcal{X} ; [0, \infty))$ is given by
\begin{align*}
\psi_\delta(x,v) =\begin{cases}
 \displaystyle{-\delta b  \, h(U(x))    \frac{v\cdot \nabla U(x)}{|\nabla U(x)|^2} }& \text{ if } U(x) \geqslant R_1 \\
 0 & \text{ otherwise}
 \end{cases}.
\end{align*}
By definition, $V_{b, \delta}\in C^{\infty}(\mathcal{X}; (0, \infty))$.  We will see at the end of the proof that, in fact, with the right choice of $\delta$ and $R_1$, $V \geqslant e^{-1}$.  Note that, as $H(x,v) \rightarrow \infty$ with $(x,v) \in \mathcal{X}$
\begin{align*}
V_{b, \delta} (x,v) = \exp\big(b H(x,v)(1+ o(1))\big)\rightarrow \infty.
\end{align*}
Next, note that for $(x,v) \in \mathcal{X}$
\begin{align}
\label{eqn:m1}\frac{L^* V_{b, \delta}(x,v)}{V_{b, \delta}(x,v)}& = - b\gamma(1-b T) |v|^2 - \delta b h(U(x)) - v \cdot \nabla_x \psi_\delta(x,v) \\
\nonumber & \qquad + (2b T-1)  \gamma \psi_\delta(x,v)+ \frac{\delta^2 b^2 \gamma Th^2(U(x)) }{|\nabla U(x)|^2}+ \gamma  b T d.
\end{align}
Note that at the very least, we need to pick $\delta > \gamma T d$ so that whenever $|v|$ is bounded and $U$ is large, the $-\delta h (U(x))$ term beats the constant $\gamma b T d$ for $U$ large.  Thus we pick $\delta = \frac{3}{2} \gamma T d$ to arrive at our chosen $V_{b}$ as in the statement.  To estimate each of the terms on the righthand side of the equation above, first observe that
\begin{align}
\nonumber  &- v \cdot \nabla_x \psi(x,v) \\
\nonumber &=  - \delta b h(U(x)) \sum_{i=1}^{d} v_i v \cdot  \partial_{x_i}\bigg( \frac{\nabla U(x)}{|\nabla U(x)|^2} \bigg) -  \delta b \sum_{i=1}^{d} v_i v \cdot  h'(U(x)) \partial_{x_i} U(x)  \frac{\nabla U(x)}{|\nabla U(x)|^2} \\
\nonumber & \leqslant \delta b  h(U(x)) \frac{|\nabla^2 U v| }{|\nabla U|^2} |v| + \delta b  |h'(U(x))| |v|^2 .
\end{align}
Hence we see that by the choices of $\delta =\frac{3}{2} \gamma T d $ and $R_1$
\begin{align*}
- v \cdot \nabla_x \psi(x,v) &\leqslant \frac{\delta b h(U(x))}{16 T d}|v|^2 + \frac{\delta b \kappa'' h(U(x))}{|\nabla U|^2} |v|^2 + \delta b \frac{|v|^2}{16 T d}  \leqslant \frac{\gamma b}{4}  | v|^2.
\end{align*}
Next, observe that by Young's inequality and the choices of $R_1>0$ and $b\in (0, 1/2T]$ we have
\begin{align*}
\gamma |2b T-1| |\psi| + \frac{\delta^2 b^2 \gamma T h^2(U)}{|\nabla U|^2} &\leqslant \frac{4 \gamma b \delta^2 h^2(U)+ \delta^2 b^2 \gamma T h^2(U) }{|\nabla U|^2} + \frac{\gamma b }{16}|v|^2\\
& \leqslant \frac{\gamma b Td}{4} + \frac{\gamma b}{16}|v|^2.
\end{align*}
Putting this into~\eqref{eqn:m1} yields the following bound
\begin{align*}
\frac{ L^* V(x,v)}{V(x,v)}& \leqslant - \frac{3b\gamma}{16} |v|^2 - \frac{3}{2}\gamma b T d   h(U(x)) + \frac{5}{4}\gamma b T d.
\end{align*}
Thus if either $|v|^2 \geqslant 8  T d$ or $U \geqslant R_2$, we arrive at the bound
\begin{align*}
\frac{L^* V(x,v)}{V(x,v)}\leqslant - \frac{\gamma b Td}{4}.
\end{align*}
Consequently, we obtain the estimate
\begin{align*}
L^* V \leqslant - \frac{\gamma b T d}{4} V + \beta' 1_{K'} \leqslant  - \frac{\gamma b T d}{4} V + \beta' 1_{K}
\end{align*}
where $K'= \{|v|^2 \leqslant 8 Td \}\cap \{ U \leqslant R_2\}$, $K$ is as in the statement of the result and $\beta' = \frac{5 \gamma b T d}{4}\max_{K'} V$.  To estimate $\beta'>0$, note that by using the choice of $R_1>0$, on $K'$ we have
\begin{align*}
|\psi_\delta| \leqslant b Td  .
\end{align*}
Hence
\begin{align*}
\beta' \leqslant \frac{5 \gamma b Td }{4} \max_{K'} \exp\bigg(\frac{b|v|^2}{2} + bU(x) + bTd \bigg)\leqslant \beta
\end{align*}
where $\beta >0$ is as in the statement of the result.  Moreover, note that, globally,
\begin{align*}
|\psi_\delta| \leqslant \frac{b \delta^2 h^2(U)}{2|\nabla U|^2} + \frac{b}{2}|v|^2.
\end{align*}
Next, note that since $b\in (0, 1/(2Td)]$
\begin{align*}
V\geqslant \exp\bigg( b U h^2(U) - \frac{b\delta^2h^2(U)}{2|\nabla U |^2}\bigg)&\geqslant \exp(b U h^2(U) - \frac{bTd}{36}h^2(U)\bigg) \geqslant e^{-1}
\end{align*}
This finishes the proof.
\end{proof}

We next combine the previous result with Theorem~\ref{thm:poincarelyap} to conclude the main general result, Theorem~\ref{thm:mainconv} with the explicit constants in the result as claimed in Corollary~\ref{cor:mainexpl}.

\begin{proof}[Proof of Theorem~\ref{thm:mainconv} and Corollary~\ref{cor:mainexpl}]
Let $\tilde{V}$ be as in the statement of Theorem~\ref{thm:LFquant}.  We recall that $\tilde{V}=e^1 V\geqslant 1$ is a strong Lyapunov function with respect to $L^*$ with constants
\begin{align*}
 \alpha = \frac{\gamma b T d}{4}, \qquad \beta = \frac{5\gamma b Td}{4}e^{bR_2+5b Td +1}
\end{align*}
and set
\begin{align*}
    K= \{ |v| \leqslant (20e^4+2) Td \} \cap \{ U \leqslant R_2\}.
\end{align*}
Also recall that $b>0$ is any constant in the interval $(0, 1/(2Td)].$  Here, however, we pick $b= \frac{1}{R_2}.$  Since $R_2=R_1 + 32 Td$, we note that this choice of $b$ is clearly in the permitted range $(0, 1/(2Td)]$.

To be able to apply Theorem~\ref{thm:poincarelyap},  we have left to check that we can pick $\lambda \geqslant 1$ so that $W=\tilde{V}+ \lambda$ and $\tilde{V}$ respectively satisfy
\begin{align}
& W(x,v) |y|^2 \geqslant (\beta \rho' +1) \bigg(\frac{\mathcal{R}(x,y)}{\gamma T} + \frac{|y|^2}{2T} \bigg), \label{eqn:pairineq}
    \\
& \tilde{V}(x,v) \geqslant \frac{2\beta}{\alpha} \frac{\mu(K^c)}{\mu(K)}, \notag
\end{align}
where $\mathcal{R}$ is as in relation~\eqref{eqn:Rdef}. We first take care of the second estimate on the right hand side above and then after finish off the first.

By Theorem~\ref{prop:quantmu}, an auxiliary estimate in the Appendix, and the values of $\alpha, \beta>0$ and $b=1/R_2$ above, we have that
\begin{align*}
\frac{2 \beta}{\alpha} \frac{\mu(K^c)}{\mu(K)} \leqslant 10 e^{3}\frac{\mu(K^c)}{\mu(K)}\leqslant 1 \leqslant \tilde{V} .
\end{align*}
Next observe that by Assumption~\ref{assump:2}, recalling that $\kappa' =1/(16Td)$ we find that
\begin{align*}
    \frac{\mathcal{R}(x,y)}{\gamma T} + \frac{|y|^2}{2T}&=  \frac{2|y|^2}{\gamma^2 T} + \frac{|\nabla^2 U(x)y|^2}{2\gamma^2 T}+ \frac{|y|^2}{2T}\\
    &\leqslant \bigg(\frac{(\kappa')^2 |\nabla U|^4+  (\kappa'')^2 +2}{\gamma^2 T}+ \frac{1}{2T} \bigg)|y|^2\\
    &\leqslant  \bigg( \frac{2 (c_0 \kappa')^2 U^{4+\frac{4}{\eta_0}}+ 2(d_0 \kappa')^2+ (\kappa'')^2 +2}{\gamma^2 T}+ \frac{1}{2T}\bigg)|y|^2\\
    &:=D(U(x)) |y|^2.
\end{align*}
Hence the inequality
\begin{align*}
    \tilde{V}(x,v)|y|^2 \geqslant (\beta \rho'+1) \bigg( \frac{\mathcal{R}(x,y)}{\gamma T}+ \frac{|y|^2}{2T}\bigg),
\end{align*}
is satisfied provided
\begin{align}
\label{eqn:Uineq}
b H(x,v)+ \psi(x,v) \geqslant \log(D(U(x)))+ \log(\beta \rho'+1) -1.
\end{align}
Recalling that
\begin{align*}
    |\psi| \leqslant \frac{b \delta^2 h^2(U)}{2|\nabla U|^2} + \frac{b}{2}|v|^2 \leqslant \frac{b}{2}|v|^2 + \frac{b Td }{36}\leqslant \frac{b}{2}|v|^2 + 1,
\end{align*}
we observe that the inequality~\eqref{eqn:Uineq} is satisfied if
\begin{align*}
    U(x) \geqslant R_2 \log(D(U(x))) + R_2 \log(\beta \rho' +1).
\end{align*}
 Pick $\lambda_0 >0$ such that $a\geqslant \lambda_0$ implies
  \begin{align*}
      a \geqslant R_2 \log(D(a)) + R_2 \log(\beta \rho'+1),
  \end{align*}
 Thus on the set $\{ U\geqslant \lambda_0 \}$, the desired inequality is satisfied.  On the compliment $\{U \leqslant \lambda_0 \}$, we have that
 \begin{align*}
     (\beta \rho' + 1) D(U(x)) \leqslant (\beta \rho' + 1) D(\lambda_0).
 \end{align*}
 Thus picking
 \begin{align*}
      \lambda \geqslant (\beta \rho' + 1) D(\lambda_0)
 \end{align*}
 ensures that the estimate~\eqref{eqn:pairineq} for $W= \tilde{V}+ \lambda$ is satisfied.  This then gives the claimed results.
 \end{proof}

\section*{Appendix}

Here we provide details behind some of the more technical estimates in the paper.

\subsection*{Quantitative inequalities for $\mu$}
Recall that $\mu$ denotes the product Gibbs measure~\eqref{eqn:Gibbs}, and let $\mu_1$ and $\mu_2$ denote the marginal measures given by
\begin{align*}
\mu_1(A) = \int_A \int_\mathcal{O}   \frac{1}{\mathcal{N}} e^{-\frac{1}{T} H(x,v)} \,dx \, dv \,\,\, \,\, \text{ and } \,\,\,\,\, \mu_2(B) = \int_B \int_{(\RR^k)^N}  \frac{1}{\mathcal{N}} e^{-\frac{1}{T} H(x,v)} \,dv \, dx
\end{align*}
defined for Borel subsets $A\subseteq (\RR^k)^N$ and $B\subseteq \mathcal{O}$.

\begin{proposition}
\label{prop:quantmu}
Suppose that $U$ satisfies Assumption~\ref{assump:1} and Assumption~\ref{assump:2} and recall the constant $R_2>0$ defined in the statement of Theorem~\ref{thm:LFquant}.  Then we have
\begin{align*}
\int_\mathcal{O} |\nabla U|^2 \, d\mu_2 \leqslant \frac{ \kappa'' T \sqrt{d}}{1-\frac{1}{16 \sqrt{d}}}\qquad \text{ and } \qquad \mu_2(\{ x \in \mathcal{O} \, : \, U(x) \geqslant R_2 \}) \leqslant  \frac{1}{2(10e^4+1)}.
\end{align*}
Moreover, if $K\subseteq \mathcal{X}$ denotes the set in Theorem~\ref{thm:LFquant}, then
\begin{align}
\mu(K^c) \leqslant \frac{1}{10e^4 +1}.
\end{align}

\end{proposition}

\begin{proof}
Suppose that $U:(\RR^k)^N\rightarrow [0, \infty]$ satisfies Assumption~\ref{assump:1} and Assumption~\ref{assump:2}.  Consider the following gradient system on $\mathcal{O}$
\begin{align}
\label{eqn:gradsys}
dX(t) = -\nabla U(X(t)) \, dt + \sqrt{2 T} \, dW(t)
\end{align}
where $W(t)$ is a standard, $Nk$-dimensional Brownian motion on $(\Omega, \mathcal{F}, \PP)$.  Under the assumptions on the potential $U$, it is not hard to show that, like equation~\ref{eqn:main}, equation~\eqref{eqn:gradsys} has unique pathwise solutions on the state space $\mathcal{O}$ for all finite times $t\geqslant 0$.  Moreover, $\mu_2$
is the unique invariant probability measure for the Markov process $X(t)$.  This can be seen by using $U$ itself as a Lyapunov functional employing the hypotheses of the statement.  For $n\in \N$, let $\xi_n$ be the first exit time of the process~\eqref{eqn:gradsys} from the set $\{ U \leqslant n \}$ and observe that
\begin{align*}
\EE_x U(X(t\wedge \xi_n))&= U(x) + \EE_x \int_0^{t\wedge \xi_n} -|\nabla U(X(s))|^2 +  T \Delta U(X(s)) \, ds .
\end{align*}
Next, to bound the quantity above, note that $|\Delta U | \leqslant  \sqrt{d} \sup_{|y|\leq 1} |\nabla^2 U y|$.  Hence applying Assumption~\ref{assump:2} gives
\begin{align*}
\EE_x U(X(t\wedge \xi_n))& \leqslant U(x) + \EE_x \int_0^{t\wedge \xi_n} -\bigg(1- \frac{1}{16 \sqrt{d}}\bigg) |\nabla U(X(s))|^2 + \sqrt{d}  \kappa'' T  \, ds.
\end{align*}
Since $U\geqslant 0$ and $\kappa' \in (0,1)$, this then implies the estimate
\begin{align*}
\EE_x  \int_0^{t\wedge \xi_n} |\nabla U(x(s))|^2 \, ds \leqslant \frac{U(x)}{1-\frac{1}{16 \sqrt{d}}} +\frac{ t \kappa'' T \sqrt{d}}{1-\frac{1}{16 \sqrt{d}}} .
\end{align*}
Using Fatou's lemma and the fact that $\xi_n\rightarrow \infty$ almost surely, we find that for all $t>0$
\begin{align*}
\frac{1}{t} \EE_x \int_0^t |\nabla U(x(s))|^2 \, ds \leqslant  \frac{U(x)}{t(1-\frac{1}{16 \sqrt{d}})} + \frac{  \kappa'' T \sqrt{d}}{1-\frac{1}{16 \sqrt{d}}}
\end{align*}
Thus, by another simple approximation argument using convergence of the C\'{e}saro means to $\mu_2$,
\begin{align*}
\int_\mathcal{O} |\nabla U|^2 \,d \mu_2 \leqslant \frac{ \kappa'' T \sqrt{d}}{1-\frac{1}{16 \sqrt{d}}}.
\end{align*}
This gives the first inequality.

For the second, observe that if $f(R)= c_\infty R^{2-\frac{2}{\eta_\infty}} - d_\infty$, then
\begin{align*}
\{ U \geqslant R\} \subseteq \{ |\nabla U|^2 \geqslant f(R)\}.
\end{align*}
Hence, employing the first inequality
\begin{align*}
\int_{\{ U \geqslant R\}} \, d\mu_2 \leqslant \int_{\{ |\nabla U|^2 \geqslant f(R)\}} \, d\mu_2\leqslant \frac{1}{f(R)} \int |\nabla U|^2 \, d\mu_2 \leqslant \frac{ \kappa'' T \sqrt{d}}{f(R)(1-\frac{1}{16 \sqrt{d}})},
\end{align*}
from which we arrive at the second claimed bound by plugging in $R=R_2$.

To obtain the final desired inequality, observe that
\begin{align*}
\mu(K^c) = \int_{K^c} \, d\mu &\leqslant \int_{\{|v|^2 \geqslant (20e^4 +2)NkT\}} \, d\mu_1 + \int_{\{ U \geqslant R_2\}} \, d\mu_2 \\
& \leqslant \int_{\{|v|^2 \geqslant (20e^4+2) NkT\}} \, d\mu_1 + \frac{1}{2(10e^4+1)}.
\end{align*}
Now, to get a bound on the remaining quantity above, this time we consider the process $V(t)$ on $(\RR^k)^N$ defined by
\begin{align*}
dV(t) = - V(t) \, dt + \sqrt{2T} \, dW(t)
\end{align*}
where $W(t)$ is a standard Brownian motion on $(\RR^k)^N$ on $(\Omega, \mathcal{F}, \PP)$.  Note that this process is exactly in the same form as in~\eqref{eqn:gradsys} by setting $U(v)= |v|^2/2$.  Hence, in exactly the same way as before, it follows that
\begin{align*}
\frac{1}{t} \int_0^t \EE_v |V(s)|^2 \, ds \leqslant \frac{|v|^2}{2t} + NkT
\end{align*}
for all $t>0$.  Consequently,
\begin{align*}
\int |v|^2 \, d\mu_1(v) \leqslant NkT.
\end{align*}
Plugging this fact into the above gives
\begin{align*}
\mu(K^c)  &\leqslant \int_{\{|v|^2 \geqslant (20e^4+2) NkT\}} \, d\mu_1 + \frac{1}{20e^4+2} \\
&\leqslant \frac{1}{(20e^4+2) NkT} \int |v|^2 \, d\mu_1(v) + \frac{1}{20e^4+2}\\
&\leqslant \frac{1}{10e^4+1}.
\end{align*}

\end{proof}

\subsection*{Quantitative bounds for singular potentials}
Next we aim to prove Proposition~\ref{prop:singularconst} which gives the claimed estimates on the singular potential in Example~\ref{ex:LJ}.  We first need the following lower bound on the gradient.
\begin{lemma}
\label{lem:singularH}
Consider the potential $U$ and the open set $\mathcal{O}$ defined in Example~\ref{ex:LJ}.  Then we have the following estimate
\begin{align}
\label{eqn:lowergrad}
| \nabla U(x)| \geqslant \frac{A a}{2 N^{3/2}} \sum_{i=1}^N |x_i|^{a -1} + \frac{Bb }{2N^{\frac{7}{2}}}\sum_{i<j} \frac{1}{|x_i-x_j|^{b+1}} -\frac{Aa}{\sqrt{N}}-B b N^{b+ \frac{5}{2}}
\end{align}
for all $x\in \mathcal{O}$.
\end{lemma}

The argument is a reworking of the proof of Lemma~4.12 of~\cite{ConradGrothaus2010}.  This proof is also in~\cite{HerzogMattingly2017}, but here we give explicit constants.
\begin{proof}
The idea behind the proof is to use the basic fact that, for $x\in \mathcal{O}$, $|\nabla U (x)| \geqslant  |\nabla U(x) \cdot y|$ for all $y\in (\RR^k)^N$ with $|y|=1$.  Then we aim to pick a convenient direction $y\in (\RR^k)^N$ with $|y|=1$.  Notationally, we set $\mathcal{Z}_N= \{ 1, 2, \ldots, N\}$.

We first claim that
\begin{align}
\label{eqn:bound1}
|\nabla U(x)| \geqslant  \frac{Aa}{\sqrt{N}} |x_i|^{a-1}- \frac{Aa}{\sqrt{N}}- Bb N^{b+\frac{5}{2}}
\end{align}
for all $i=1,2, \ldots, N$ and $x\in \mathcal{O}$.  From this, summing both sides of the previous inequality from $1$ to $N$ it follows that
\begin{align}
\label{eqn:bound2}
|\nabla U (x)| \geqslant \frac{Aa}{N^{3/2}}\sum_{i=1}^N |x_i|^{a-1} - \frac{Aa}{\sqrt{N}}- Bb N^{b+ \frac{5}{2}}
\end{align}
on $\mathcal{O}$. Note that without loss of generality it suffices to show the bound~\eqref{eqn:bound1} above for $i=1$ and for $|x_1| > 1$.  For $x=(x_1, x_2, \ldots, x_N)\in \mathcal{O}$, consider an increasing sequence of sets $S_i(x)$, $i=1,2, \ldots, N$, defined inductively as follows:
\begin{align*}
S_1(x) &= \{ j \in \mathcal{Z}_N\, : \, |x_1- x_j| < N^{-1}\} \\
S_m(x) &= \{ j \in \mathcal{Z}_N \, : \, |x_j - x_k|< N^{-1} \,\, \exists \,\,k \in S_{m-1}(x)\}, \,\,\,\, m =2, \ldots, N.
\end{align*}
Observe that $S_1(x)\neq \emptyset$ since $1\in S_1(x)$.  Also note that for any $i,j \in S_N(x)$, $|x_i - x_j| < 1$.  Consequently, combining $|x_1-x_j|^2=|x_1|^2+ |x_j|^2 - 2 x_1\cdot x_j$ with the inequality $|x_1-x_j|<1$, it follows that for $j\in S_N(x)$, $2 x_1 \cdot x_j \geqslant |x_1|^2 + |x_j|^2 - 1 \geqslant 0$
where the last inequality follows since $|x_1|>1$.  Moreover, if $i\in S_N(x)$ while $j\notin S_N(x)$ we have that $|x_i - x_j | \geqslant N^{-1}$.  Let $\sigma(x)=(\sigma_1(x), \ldots, \sigma_N (x)) \in (\RR^k)^N$ be such that $\sigma_i(x)= x_1/|x_1|$ if $i\in S_N(x)$ and $\sigma_i=0$ otherwise.  We thus have the bound
\begin{align}
\nonumber &\sqrt{N} |\nabla U(x)| \\
\nonumber &\geqslant \sigma(x) \cdot \nabla U(x)\\
\nonumber &= A a \sum_{n \in S_N(x)}  \frac{x_1\cdot x_n}{|x_1|}|x_n|^{a-2}  + Bb \sum_{n\in S_N(x)} \sum_{\substack{i=1\\i\neq n}}^N  \frac{x_1 |x_1|^{-1}\cdot (x_n - x_i)}{|x_n-x_i|^{b+2}}\\
\label{eqn:lastbound}&\geqslant Aa  |x_1|^{a-1} +Bb \sum_{n\in S_N(x)} \sum_{\substack{i=1\\i\neq n\\i \in S_N(x)}}^N  \frac{x_1 |x_1|^{-1}\cdot (x_n - x_i)}{|x_n-x_i|^{b+2}}
\\
\nonumber & \qquad +  Bb\sum_{n\in S_N(x)} \sum_{\substack{i=1\\i\neq n\\i \notin S_N(x)}}^N  \frac{x_1 |x_1|^{-1}\cdot (x_n - x_i)}{|x_n-x_i|^{b+2}}.
\end{align}
Next note that
\begin{align*}
\sum_{n\in S_N(x)} \sum_{\substack{i=1\\i\neq n\\i \in S_N(x)}}^N  \frac{x_1 |x_1|^{-1}\cdot (x_n - x_i)}{|x_n-x_i|^{b+2}} &= \sum_{\substack{n,i \in S_N(x)\\n\neq i}} \frac{x_1 |x_1|^{-1}\cdot (x_n - x_i)}{|x_n-x_i|^{b+2}} \\
&= \sum_{\substack{n,i \in S_N(x)\\n\neq i}} \frac{x_1 |x_1|^{-1}\cdot (x_i - x_n)}{|x_n-x_i|^{b+2}}.\end{align*}
Consequently,
\begin{align*}
\sum_{n\in S_N(x)} \sum_{\substack{i=1\\i\neq n\\i \in S_N(x)}}^N  \frac{x_1 |x_1|^{-1}\cdot (x_n - x_i)}{|x_n-x_i|^{b+2}} =0.
\end{align*}
Plugging this back into~\eqref{eqn:lastbound}, we obtain the following inequality
\begin{align*}
\sqrt{N} |\nabla U(x)|&\geqslant Aa  |x_1|^{a-1} + 0 - Bb N^{b + 3}.
\end{align*}
This finishes the proof of the bound~\eqref{eqn:bound1} when $i=1$, as desired.

We next show that for all $i,m \in \mathcal{Z}_N$ with $i\neq m$
\begin{align}
\label{eqn:bound3}
|\nabla U(x)| \geqslant \frac{2B b}{\sqrt{N}}|x_i-x_m|^{-b-1} - \frac{A a}{\sqrt{N}}\sum_{n=1}^N |x_n|^{a-1}
\end{align}
for all $x =(x_1, x_2, \ldots, x_N) \in \mathcal{O}$.  Note then that this estimate together with the bound~\eqref{eqn:bound2} implies the lemma.  Without loss of generality, we will prove the bound~\eqref{eqn:bound3} for $i=1, m=2$.  For $x\in \mathcal{O}$, let $\sigma(x)= (x_2-x_1)/|x_2-x_1|$ and $\xi_k(x)= c_k(x) \sigma(x)$ where the constants $c_k(x) \in \{-1,1\}$ are chosen to satisfy $c_k(x) = 1$ if $x_k\cdot \sigma(x) < x_2 \cdot \sigma(x)$ and $c_k(x)=-1$ otherwise.  With this choice of direction $\xi(x):=(\xi_1(x), \ldots, \xi_N(x))$ we find that on $\mathcal{O}$:\begin{align}
\label{eqn:bound4}
\nonumber \sqrt{N} |\nabla U(x)|& \geqslant \xi(x) \cdot \nabla U(x)\\
\nonumber & = \sum_{n =1}^N A a  \xi_n(x) \cdot x_n |x_n|^{a-2}  + B b\sum_{i <n}   (\xi_i(x)-\xi_n(x)) \cdot \frac{x_n-x_i}{|x_n-x_i|^{b+2}}\\
\nonumber &\geqslant -  Aa \sum_{n=1}^N |x_n|^{a-1}+ Bb  \sum_{i <n}   (\xi_i(x)-\xi_n(x)) \cdot \frac{x_n-x_i}{|x_n-x_i|^{b+2}}\\
&= -  Aa \sum_{n=1}^N |x_n|^{a-1} +Bb \sum_{i <n}   (c_i(x)-c_n(x)) \cdot \frac{x_n \cdot \sigma(x)-x_i\cdot \sigma(x)}{|x_n-x_i|^{b+2}}.
\end{align}
To bound the remaining term on the righthand side above, note that if $i, n$ are either such that $x_i \cdot \sigma(x) < x_2(x) \cdot \sigma(x)$ and $x_n \cdot \sigma(x) < x_2\cdot \sigma(x)$ or such that $x_i \cdot \sigma(x) \geqslant x_2 \cdot \sigma(x)$ and $x_n \cdot \sigma(x) \geqslant x_2 \cdot \sigma(x)$, then $c_i(x)=c_n(x)$.  Hence the corresponding term in the sum in~\eqref{eqn:bound4} is zero.
On the other hand, if $i,n$ are either such that $x_i\cdot \sigma < x_2 \cdot \sigma \leqslant x_n \cdot \sigma$ or such that $x_n \cdot \sigma < x_2\cdot \sigma \leqslant x_i \cdot \sigma$, then the corresponding term in the sum~\eqref{eqn:bound4} is nonnegative via the $c_i(x)$, $c_n(x)$.  In particular, by these observations we arrive at the estimate
\begin{align}
\label{eqn:bound4prime}
\nonumber \sqrt{N} |\nabla U(x)|& \geqslant -  Aa \sum_{n=1}^N |x_n|^{a-1} + Bb (c_1(x)-c_2(x)) \frac{(x_2 -x_1) \cdot \sigma}{|x_2-x_1|^{b +2}}\\
&=  -  A a \sum_{n=1}^N |x_n|^{a-1} + 2 Bb \frac{1}{|x_2-x_1|^{b +1}}\end{align}
as $c_1(x)-c_2(x) =2$ and $(x_2-x_1)\cdot \sigma = |x_2-x_1|$ by construction.
\end{proof}

We can now use the previous result to conclude Proposition~\ref{prop:singularconst}.
   \begin{proof}[Proof of~\ref{prop:singularconst}]
We begin by computing $\nabla U$ and $\nabla^2U$ on $\mathcal{O}$.  Observe that for $x\in\mathcal{O}$ and $i,n=1,2, \ldots, N$, $\ell,m =1,2,\ldots, k$,
\begin{align*}
\partial_{x_i^\ell} U(x)&= A a x_i^\ell |x_i|^{a-2} - B b \sum_{\substack{j=1\\j\neq i}}^N \frac{x_i^\ell -x_j^\ell}{|x_i -x_j|^{b+2}}
\end{align*}
and
\begin{align*}
\partial_{x_{n}^m  x_i^\ell}^2 U(x)&= Aa \delta_{(i, \ell), (n, m)} |x_i|^{a-2} + A a(a-2) \delta_{i, n}  |x_i|^{a-4} x_i^\ell x_{n}^m\\
& \hspace{-0.65in}+ B b\sum_{\substack{j=1\\j\neq i}}^N \bigg\{\frac{\delta_{(j, \ell), (n, m)} - \delta_{(i, \ell), (n, m)} }{|x_i -x_{j}|^{b+2}} + (b+2) \frac{(x_i^\ell-x_{j}^\ell)(\delta_{i, n}(x_i^m-x_{j}^m) + \delta_{n, j} (x_{j}^m-x_i^m))}{|x_i -x_{j}|^{b+4}}\bigg\}
\end{align*}
where $\delta_{i,j}=1$ if $i=j$ and $0$ otherwise.  Also, the $(a-2)\times \cdots$ term is defined to be $0$ when $a=2$.  Thus for any $y\in (\RR^k)^N$ with $|y| \leqslant 1$ and any $x\in \mathcal{O}$ we arrive at the estimate
\begin{align}
\label{eqn:HessSingH}
|\nabla^2 U(x) y|&\leqslant Aa(a-1)k \sum_{i=1}^N |x_i|^{a-2} + 4Bb (b+3) k \sum_{i<n} \frac{1}{|x_i -x_{n}|^{b+2}}.
\end{align}
Applying Lemma~\ref{lem:singularH} with Young's inequality gives that for all $x\in \mathcal{O}$
\begin{align}
\label{eqn:GradSingH}
|\nabla U(x)|^2 \geqslant \frac{A^2a^2}{8 N^{3}} \sum_{i=1}^N |x_i|^{2a -2} + \frac{B^2 b^2 }{8 N^7}\sum_{i<j} |x_i-x_j|^{-2b-2} -\frac{2A^2a^2}{N}-2 B^2b^2 N^{2b+5}.
\end{align}
In order to compare~\eqref{eqn:GradSingH} with~\eqref{eqn:HessSingH}, next note that for any constants $C_i>0$ we have
\begin{align*}
\sum_{i=1}^N |x_i|^{a-2} &= \sum_{\{i\, :\,  |x_i|^{a} \geqslant C_1\}} |x_i|^{a-2} + \sum_{\{i: |x_i|^{a} \leqslant C_1\}} |x_i|^{a-2} \leqslant C_1^{\frac{a-2}{a}} N+ \frac{1}{C_1}\sum_{i=1}^N |x_i|^{2a-2}
\end{align*}
and, similarly,
\begin{align*}
\sum_{i<j} \frac{1}{|x_i -x_j |^{b+2}}&=\sum_{\{i<j\, :\,  |x_i-x_j|^{b} \leqslant 1/C_2\}} \frac{1}{|x_i -x_j |^{b+2}} + \sum_{\{i<j\, :\,  |x_i-x_j|^{b} \geqslant 1/C_2\}} \frac{1}{|x_i -x_j |^{b+2}} \\
& \leqslant C_2^{\frac{b+2}{b}}N^2 +\frac{1}{C_2} \sum_{i<n} \frac{1}{|x_i -x_n|^{2b+2}}.
\end{align*}
Thus picking $C_1=128(a-1)N^4k^2 T/(Aa)$, $C_2=512 N^8k^2 (b+3)T/(Bb)$ and combining~\eqref{eqn:GradSingH} with~\eqref{eqn:HessSingH} produces the following bound on $\mathcal{O}$
\begin{align}
|\nabla^2 U (x) \cdot y| \leqslant \frac{1}{16Td}|\nabla U(x)|^2 + \kappa''
\end{align}
for any $y \in (\RR^k)^N$ with $|y| \leqslant 1$ where $\kappa''$ is as in the statement of the result.  Note that this validates the first part of Assumption~\ref{assump:2}.

In order to check the second part of Assumption~\ref{assump:2}, first note that for any $\eta_0> -1, \eta_0 \neq 0,$ we have the following estimate
\begin{align}
\label{eqn:singUest1}
U(x)^{2+\frac{2}{\eta_0}} \geqslant \frac{1}{2N} \sum_{i=1}^N A^{2+ \frac{2}{\eta_0}} |x_i|^{2a+\frac{2a}{\eta_0}} + \frac{1}{2N^2} \sum_{i<j} \frac{B^{2+\frac{2}{\eta_0}}}{|x_i -x_j|^{2b+\frac{2b}{\eta_0}}}
\end{align}
for any $x\in \mathcal{O}$.  Moreover, for any $\eta_\infty>1$ and any $x\in \mathcal{O}$
\begin{align}
\label{eqn:singUest2}
U(x)^{2-\frac{2}{\eta_\infty}} \leqslant (2N)^{2-\frac{2}{\eta_\infty}} \sum_{i=1}^N A^{2-\frac{2}{\eta_\infty}} |x_i|^{2a- \frac{2a}{\eta_\infty}} + (2N^2)^{2- \frac{2}{\eta_\infty}} \sum_{i<j} \frac{B^{2-\frac{2}{\eta_\infty}}}{|x_i -x_j|^{2b-\frac{2b}{\eta_\infty}}}
\end{align}
Also note that on $\mathcal{O}$
\begin{align}
\label{eqn:singGradest2}
|\nabla U(x)|^2 &\leqslant 2 A^2 a^2 \sum_{i=1}^N |x_i|^{2a-2} + 2 B^2 b^2 N \sum_{i<j} \frac{1}{|x_i-x_j|^{2b+2}}
\end{align}
Combining the estimate~\eqref{eqn:singUest1} with~\eqref{eqn:singGradest2} and picking $\eta_0=b$ produces the inequality
\begin{align}
|\nabla U(x)|^2 \leqslant c_0 U(x)^{2+\frac{2}{b}} + d_0
\end{align}
where $c_0$ and $d_0$ are as in the statement of the result.  Similarly, combining~\eqref{eqn:singUest2} with~\eqref{eqn:GradSingH} and choosing $\eta_\infty=a$ gives
\begin{align}
|\nabla U(x)|^2 \geqslant c_\infty U(x)^{2-\frac{2}{a}} - d_\infty
\end{align}
where $c_\infty$ and $d_\infty$ are as in the statement of the result.

\end{proof}

\providecommand{\bysame}{\leavevmode\hbox to3em{\hrulefill}\thinspace}
\providecommand{\MR}{\relax\ifhmode\unskip\space\fi MR }
\providecommand{\MRhref}[2]{%
  \href{http://www.ams.org/mathscinet-getitem?mr=#1}{#2}
}
\providecommand{\href}[2]{#2}

\end{document}